\documentclass[10pt]{article}
\usepackage{amsmath}
\usepackage{amsfonts}
\usepackage{amssymb}
\usepackage{amsthm}
\usepackage{mathtools}
\usepackage{algorithmic, algorithm}
\usepackage{graphicx}
\usepackage[auth-sc]{authblk}
\usepackage[utf8]{inputenc}
\usepackage[english]{babel}
\usepackage{geometry}
\usepackage{cite}
\usepackage{xcolor}
\newtheorem{theorem}{Theorem}[section]
\newtheorem{proposition}[theorem]{Proposition}

\newtheorem{lemma}[theorem]{Lemma}

\newtheorem{assumption}[theorem]{Assumption}
\newtheorem{definition}[theorem]{Definition}
\usepackage{subcaption}

\newcommand{\R}{\mathbb R}

\newcommand{\f}[1]{\mathbf{#1}}
\newcommand{\norm}[1]{\left\| #1 \right\|}
\newcommand{\abs}[1]{\lvert#1\rvert}
\newcommand{\set}[1]{\left\{ #1\right\}}
\newcommand{\lapla}{\Delta}
\newcommand{\manifold}{\mathcal{M}}

\numberwithin{equation}{section}

 \graphicspath{{./fig/}}

\begin{document}
\title{Superconvergence of differential structure for finite element methods on perturbed surface meshes\thanks{Emails: guozhi.dong@csu.edu.cn; hailong.guo@unimelb.edu.au; guoting@hunnu.edu.cn}}
\author[$\dagger$]{Guozhi Dong}
\author[$\ddagger$]{Hailong Guo}
\author[$\star$]{Ting Guo}
\affil[$\dagger$]{School of Mathematics and Statistics, HNP-LAMA, Central South University, Changsha 410083, China}
\affil[$\ddagger$]{School of Mathematics and Statistics, The University of Melbourne, Parkville, VIC, 3010, Australia}
\affil[$\star$]{Key Laboratory of Computing and Stochastic Mathematics (Ministry of Education), School of Mathematics and Statistics, Hunan Normal University, Changsha 410081, China}
\date{}
\maketitle
\begin{abstract}
Superconvergence of differential structure on discretized surfaces is studied in this paper. The newly introduced geometric supercloseness provides us with a fundamental tool to  prove the superconvergence of gradient recovery on deviated surfaces. An algorithmic framework for gradient recovery without exact geometric information is introduced.  Several numerical examples are documented to validate the theoretical results.

	\vskip .3cm
	
\end{abstract}


\section{Introduction}
\label{sec:into}

The numerical solution of partial differential equations on surfaces or more general surfaces has been the subject of much systematic investigation.  Many efficient numerical methods have been developed since the pioneering work of Dziuk \cite{Dziuk1988}.  However, their numerical analysis including the {\it a priori} error analysis of surface finite element methods \cite{DziukElliott2013, BDN} and the {\it a posteriori} error analysis \cite{WeiChenHuang2010, DonGuo20} usually requires exact information of the surfaces. For instance, many methods ask that the vertices of discrete surfaces  are located on the underlying surfaces and the exact unit normal vectors at the given vertices are known. This is neither theoretically complete nor practically available since the exact geometric information is often blind to users in reality.
Therefore, it is of interest and also practically meaningful to investigate the problems where the exact geometric information is not given.
We particularly pay attention to the cases when solutions contain differential structures of the surfaces. First-order differential structures involve tangential spaces and normal spaces of the surfaces, while the former is our focus in this paper. Typical examples are tangential vector fields on surfaces and gradients of scalar functions on surfaces. In such situations, it is desired to know the conditions for geometric discretization to guarantee optimal convergence rates either in the a priori or the a posteriori error analysis.
Fundamental questions here are like (i) to what extent that the errors of geometric approximations will affect the total errors of the numerical methods, and (ii) what is the hypothesis on the geometric discretization in order to have optimal convergence of numerical solutions or superconvergence of the differential structure of surfaces.

This paper aims to provide some insight into these questions. Such problems have been open in the community for a while. 
For instance in  \cite{WeiChenHuang2010}, gradient recovery schemes on general surfaces have been systematically investigated, and superconvergence rates of several recovery schemes were proven provided that the exact geometry is given. The supercloseness of the numerical data has played a crucial role in establishing the theoretical results in \cite{WeiChenHuang2010}.
Moreover, the following two interesting questions arose in \cite{WeiChenHuang2010}: (i) How to design gradient recovery algorithms given no exact information of the surfaces (i.e., no exact normal vectors and no exact vertices)? (ii) Is it possible to preserve the superconvergence rates of gradient recovery schemes using triangulated meshes whose vertices are not located on the exact surfaces but in a $\mathcal{O}(h^2)$ neighborhoods of the underlying surfaces? Here $h$ is the scale of the mesh size.

These questions partially motivate the research  here. In particular, \emph{superconvergence of gradient recovery on surfaces} is  connected to the concept of \emph{geometric supercloseness} which we propose in this paper.
Gradient recovery techniques for data defined in Euclidean domain have been intensively investigated \cite{BankXu2003,XuZhang2004,AinsworthOden2000,Lakhany2000,ZhangNaga2005,ZZ1992,ZZ1992b, GXZ2019, GZZ2017}, and also find many interesting applications, e.g. \cite{NagaZhangZhou2006,NagaZhang2012, GZZ2018, CGZZ2017}.
The methods for data on discretized surfaces have been studied, e.g., in \cite{DuJu2005, WeiChenHuang2010}.
Using the idea of tangential projection,  many of the recovery algorithms in the setting of the Euclidean domain have been generalized to the setting of surfaces.
However, there are certain restrictions in the existing approaches as many of them require the exact geometry (exact vertices, exact normal vectors) either for designing algorithms or for proving superconvergent rates.
In \cite{DonGuo20} a novel gradient recovery scheme for data defined on discretized surfaces was proposed, which is called the parametric polynomial preserving recovery (PPPR) method. PPPR does not rely on the exact geometry-prior, and it was proven to be able to achieve superconvergence under mildly structured meshes, including high curvature cases. That can be thought of partially answered the first open question in \cite{WeiChenHuang2010}.
However, the theoretical proof for the superconvergence result in \cite{DonGuo20} still requires that the vertices are located on the underlying exact surfaces, though numerically the superconvergence has been observed when this condition is violated.

\paragraph{Contribution}
In this paper, we first construct some examples to show that there exist cases where the superconvergence of gradient recovery on surfaces is not guaranteed given barely the $\mathcal{O}(h^2)$ vertex condition. In particular, the examples show that data supercloseness does not guarantee superconvergence of the recovered gradient, in contrast to the exact nodal points case.
We introduce a new concept called \emph{geometric supercloseness}, which gives the property of superconvergence of differential structure on deviated discretizations of surfaces.
Especially, we provide conditions of the discretized meshes under which the geometric supercloseness property can be proven. 
With the tool of  superconvergence of differential structure, we provide complete answers to the two open questions in \cite{WeiChenHuang2010}.  To do this, we generalize the idea from \cite{DonGuo20}. That is the idea of local parametric polynomials can be further developed to cover other methods, e.g., superconvergence patch recovery (SPR), of which their counterparts using exact geometric information have been discussed in \cite{WeiChenHuang2010}. In this vein, we develop an isoparametric SPR schemes for data on discretized surfaces.
It consists of two-level recoveries: The first is recovering the Jacobian of local geometric mapping over every parametric domain, and the other is iso-parametrically recovering the gradient of the solution function with respect to the parametric arguments.
Based on such a two-level scheme, we are able to prove that the superconvergence of the resulting gradient recovery, which require the superconvergence of the differential structure of the surface as well as the superconvergence of the isoparametric gradient of the function data (i.e., numerical solutions).
 
\paragraph{Structure}
The rest of the paper is organized as follows:
In Section \ref{sec:geometric}, the general geometric setting and notations are explained, and some counter examples on the superconvergence of gradient recovery with arbitrary $\mathcal{O}(h^2)$ deviated vertex condition are provided.
In Section \ref{sec:geometric_superclose}, the concept of geometric supercloseness is introduced and the precise hypothesis on geometric discretization is provided for proving the superconvergence of differential structure on deviated surfaces.
Section \ref{sec:app_1} proves the superconvergence of gradient recovery scheme on deviated surfaces given the geometric supercloseness condition.
In Section \ref{sec:numerics}, we show numerical examples which verify the theoretical findings. 
Some conclusive remarks are given in Section \ref{sec:conclusion}.

\section{Geometric setting and counter examples}
\label{sec:geometric}
We start this section by specifying some of the geometrical notations which are frequently referred in the paper. Then we  provide a counter example to show that the superconvergence of gradient recovery is not guaranteed under general $\mathcal{O}(h^2)$ perturbation of vertices.
\subsection{Geometric setting}
$\manifold$ is a general two dimensional $C^3$ smooth compact hypersurface embedded in $\R^{3}$ which is endowed with a Riemann metric $g$, and $\manifold_h=\bigcup_{j\in J_h} \tau_{h,j}$ is a triangular approximation of $\manifold$, with $h=\max_{j\in J_h} \mbox{diam}(\tau_{h,j})$ being the maximum diameter of the triangles $\tau_{h,j}$.
Here $J_h$ and $I_h$ are the index sets for triangles and vertices of $\manifold_{h}$, respectively.
We denote  $\set{\tau_j}_{j\in J_h}$ the corresponding curved triangles which satisfy $\bigcup_{j\in J_h} \tau_j= \manifold$.
Note that the vertices of $\manifold_h$ do not necessarily locate on $\manifold$, therefore $\tau_j$ and $\tau_{h,j}$ may have no common vertices.
In the following study, we introduce $\manifold^*_h$ to be the counterpart of $\manifold_h$ with the same number of vertices, all of which are located on $\manifold$. 
To obtain $\manifold^*_h$, we project $\set{x_{h,i}}_{i\in I_h}$ the vertices of $\manifold_h$ along unit normal direction of $\manifold$ to have $\set{x^*_{h,i}}_{i\in I_h}$ the vertices of $\manifold_h^*$. 
Then we connect $\set{x^*_{h,i}}_{i\in I_h}$ using the same order as the connection of  $\set{x_{h,i}}_{i\in I_h}$, which gives the triangulation of $\manifold^*_h$.
Then $\set{\tau^*_{h,j}}_{j\in J_h} $ denote the corresponding triangles on $\manifold^*_h$.
To illustrate the main idea, we focus on  the linear surface finite element method \cite{Dziuk1988}.  In that case,  the nodal points  simply consist of  all the vertices of $\manifold_h$. 

In \cite{WeiChenHuang2010, DuJu2005}, gradient recovery methods have been generalized from planar domain to surfaces, while they are restricted to the case that the vertices are located on the underlying exact surface. In other words, they have been only studied in the case that the discretization is given by, corresponding to our notation, $\manifold_h^*$.
It has been, however, conjectured that the superconvergence of gradient recovery on general discretized surfaces, like $\manifold_h$, may be proven if the vertices of $\manifold_h$ are in a $\mathcal{O}(h^2)$ neighborhood of the corresponding vertices of $\manifold^*_h$. That is the following vertex-deviation condition
\begin{equation}
\label{eq:vertex_error}
\abs{x^*_{h,i}-x_{h,i}}= \mathcal{O}(h^2) \quad   \text{ for all } i\in I_h.
\end{equation}

We recall the transform operators between the function spaces on $\manifold$ and on $\manifold_h$ (or similarly $\manifold_h^*$). Let $\mathcal{V}(\manifold)$ and $\mathcal{V}(\manifold_h)$ be some ansatz function spaces. We define the  following transform operators
\begin{equation}
\label{eq:transform}
\begin{aligned}
T_h : \mathcal{V}(\manifold) &\to  \mathcal{V}(\manifold_h); \\
v &\mapsto  v \circ P_h,
\end{aligned}\quad
\text{
	and 
}\quad
\begin{aligned}
(T_h)^{-1}:  \mathcal{V}(\manifold_h) &\to  \mathcal{V}(\manifold); \\
v_h &\mapsto   v_h \circ P_h^{-1},
\end{aligned}
\end{equation}
where $P_h:\set{\tau_{h,j}}_{j\in J_h}\to \set{\tau_{j}}_{j\in J_h}$ is a bijective map to have the correspondence between $\set{\tau_{h,j}}_{j\in J_h}$ and $\set{\tau_j}_{j\in J_h}$.
The transform operators $(T^*_h)^\pm$  between functions on $\manifold$ and $\manifold^*_h$ can be   defined similarly. 
Note that in the following, we may abuse a bit of notation for vector valued functions, i.e., we still use the same $T_h$ (or $T^*_h$) for vector valued functions, in which case $T_h$ (or $T^*_h$) is applied to each component of the function.

In the following analysis, for each vertex $x^*_{h,i}$, a local parametrization function $\f r_i:\Omega_i \to \manifold$ is needed, which maps an open set in the parameter domain $\Omega_i \subset \R^2$ to an open set contains $x_{h,i}^*$ on the surface. Note that we take $\Omega_i$ a compact set which is the parameter domain corresponding to the selected patch on $\manifold^*_h$ around the vertex $x^*_{h,i}$, or respectively the patch on $\manifold_h$ around the vertex $x_{h,i}$.
In such a way, we define local parametrization functions $\f r_{h,i}:\Omega_i \to \manifold_h$ and $\f r_{h,i}^* :\Omega_i \to \manifold_h^*$, respectively. They are piecewise linear maps.
In addition, we use $\f r_{\tau_{h,j}}:\tau_{h,j}\to \tau_j $ and $\f r^*_{\tau_{h,j}}:\tau^*_{h,j}\to \tau_j $ to denote the local parameterizations from the small triangle pairs $\tau_{h,j}$ and $\tau^*_{h,j}$ to $\tau_j$, respectively.
Due to the smoothness assumption on $\manifold$, $\f r_i$, $\f r_{\tau_{h,j}}$ and $\f r^*_{\tau_{h,j}}$   are $C^3$ continuous for every $i\in I_h$ and $j\in J_h$. It follows that  $\f r_i\in W^{3,\infty}(\Omega_i)$ and $\f r_{\tau_{h,j}} \in W^{3,\infty}(\tau_{h,j})$ and  $\f r^*_{\tau_{h,j}} \in W^{3,\infty}(\tau^*_{h,j})$.
The condition \eqref{eq:vertex_error} indicates that triangulated surface $\manifold_h$ converges to $\manifold$ as $h\rightarrow 0$.

\begin{figure} [!h]
	\centering
	\includegraphics[width=0.6\textwidth]{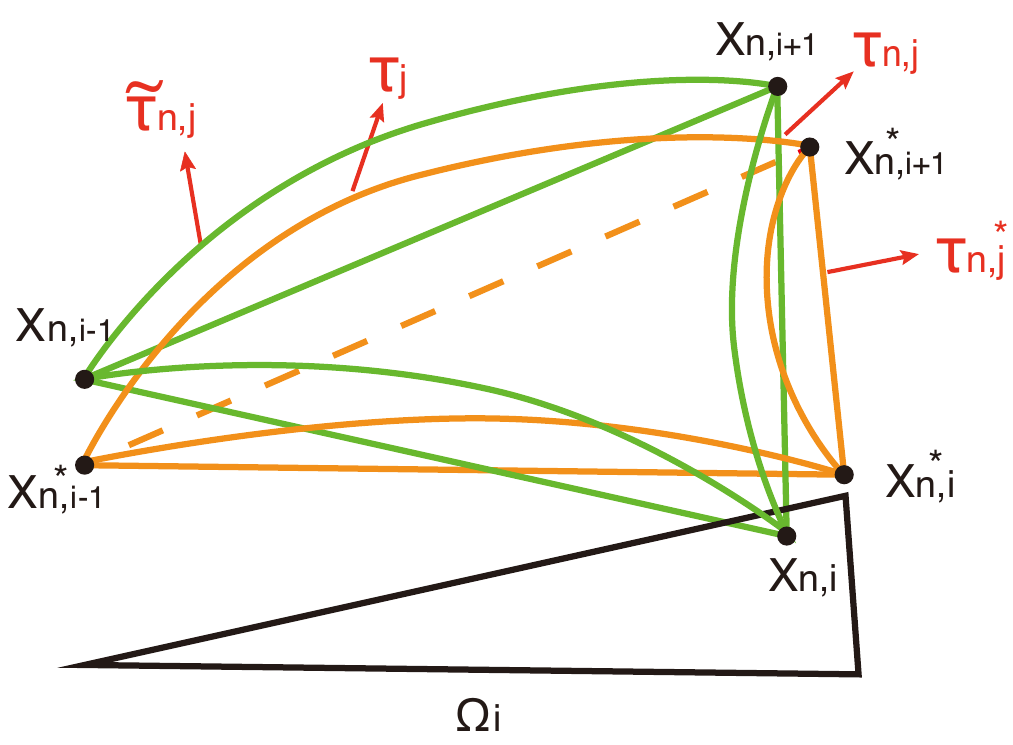}
	\caption{Here $\tau_{h,j}$ and $\tau^*_{h,j}$ are planar triangles, while $\tau_j$ and $\tilde{\tau}_{h,j}$ are curved triangles, $\Omega_i$ is the parameter domain for the surrounding patches at $x^*_{n,i}$, which we show only one triangle part of it.
 This parameter domain is used to define the local geometric maps for triangles surrounding the vertex $x^*_{n,i}$ , i.e., $\f r_i:\Omega_i \to  \tau_j$, $\f r_{h,i}:\Omega_i \to \tau_{h,j}$, $\f r_{h,i}^* :\Omega_i \to  \tau^*_{h,j}$. On the planar triangles we have the local maps $\f r_{\tau_{h,j}}:\tau_{h,j}\to \tau_j $, $ \f r^*_{\tau_{h,j}}: \tau^*_{h,j}\to \tau_j$ and $\tilde{\f r}_{\tau_{h,j}}:\tau_{h,j}\to \tilde{\tau}_{h,j} $.}
	\label{fig:surfaces}
\end{figure}
In Figure \ref{fig:surfaces}, we use one set of triangle patches indexed by some $j\in J_h$ to help illustrating the relations. There $\tau_j\subset \manifold$, $\tau^*_{h,j}\subset \manifold_h^*$, $\tau_{h,j}\subset \manifold_h$, and $\tilde{\tau}_{h,j} \subset \widetilde{\manifold}_h$, are corresponding to one triangle face on the exact surface, the patch-wise linear interpolation of the exact surface, an approximation of the exact surface, and a higher-order approximation of the exact surface (see Proposition \ref{prop:appr_surfaces} in Section \ref{sec:app_1}), respectively.

\subsection{Examples of $\mathcal{O}(h^2)$ deviated surfaces where superconvergence fails}
In this subsection, we present numerical examples which motivate further study in this paper.
They also illustrate why the concept of geometric supercloseness is needed. 
Especially, we show that barely with the condition \eqref{eq:vertex_error}, superconvergence rates of the recovered gradient may not be achieved as conjectured in \cite{WeiChenHuang2010}. 
Without loss of generality, we test the Laplace-Beltrami equation whose exact solution is $u = x_1x_2$ on the unit sphere.
For the discretization, we use uniform triangulation with nodes located on the sphere to get  $\manifold_{h}^*$, and then add $\mathcal{O}(h^2)$  random perturbation to the vertices of $\manifold_h^*$ in the tangential direction to get $\manifold_h^1$ and in the normal direction to get $\manifold_h^2$. To test the superconvergence property of the recovered gradient, we solve the Laplace-Beltrami equation on $\manifold_h^i$ ($i=1,2$) and  the numerical results are summarized in  Figure \ref{fig:cex}. We use PPPR scheme \cite{DonGuo20} for recovery which has been shown to have superconvergence under mildly structured mesh conditions with exact interpolation of the geometry:
\begin{equation}\label{eq:data_close}
\norm{ \nabla_{g_h} u_I -  \nabla_{g_h} u_h}_{0,\manifold_h^i} =\mathcal{O}(h^2) \quad \text{ for }\;  i=1,2,
\end{equation}
where $u_I$ is the linear interpolation of exact solution,  $u_h$ is the finite element solution, and $\norm{ \cdot}_{0,\manifold_h^i}$ denotes the $L^2$ norm on $\manifold_h^i$
However, there is no superconvergence observed for the recovered gradient on those $\mathcal{O}(h^2)$ deviated meshes. 
In \cite{WeiChenHuang2010}, it has been proven that if the discretization is given by exact interpolation of the geometry, the supercloseness  \eqref{eq:data_close} leads to the superconvergence of the recovery of gradient under some shape conditions on the triangulation.
The example here delivers  the message that if the discretized geometry is $\mathcal{O}(h^2)$ deviated, then \eqref{eq:data_close} is not sufficient anymore to guarantee the superconvergence of the recovered gradient. Instead, we need the superconvergence of differential structure of the surfaces.
Motivated by this, we  shall investigate the geometric conditions for superconvergence in post-processing numerical solutions.

\begin{figure} [!h]
	\centering
	\subcaptionbox{\label{fig:cex_tan}}
	{\includegraphics[width=0.47\textwidth]{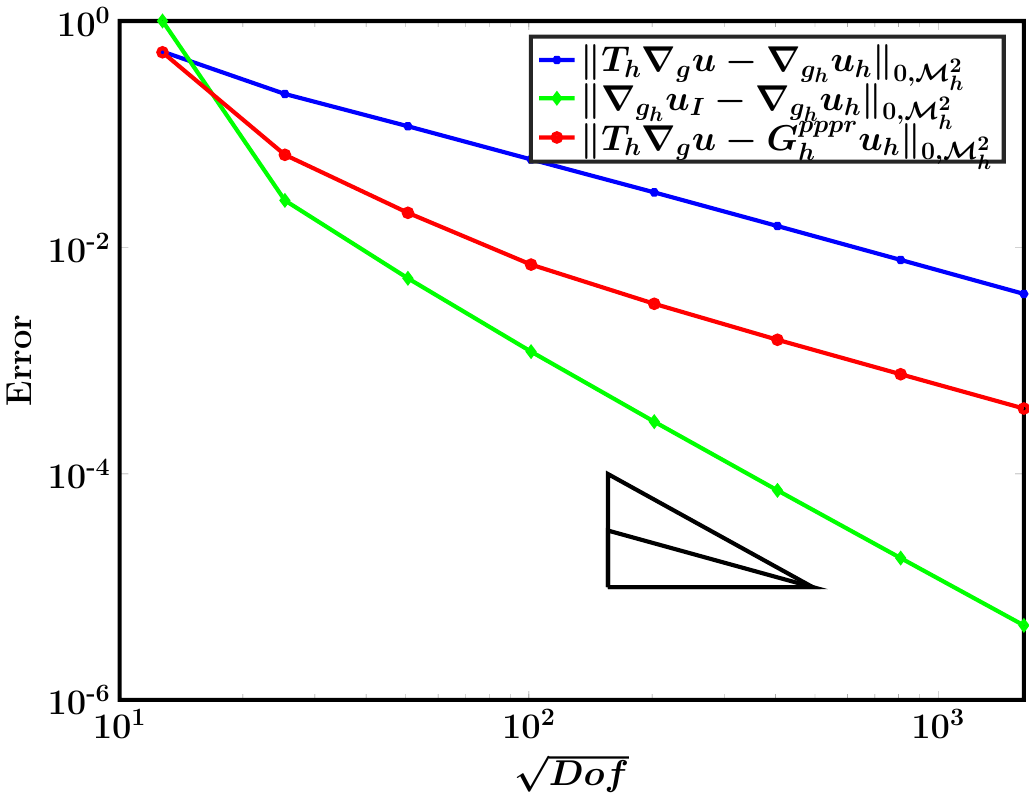}}
	\subcaptionbox{\label{fig:cex_norm}}
	{\includegraphics[width=0.47\textwidth]{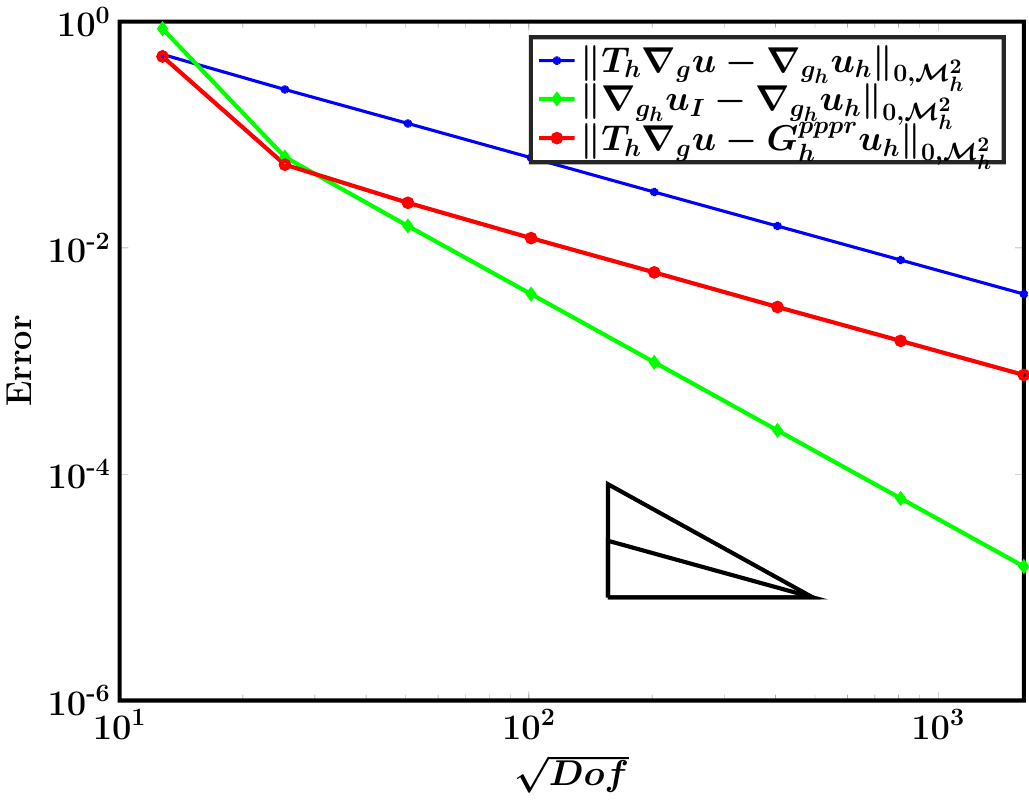}}
	\caption{Counter example of superconvergence of the recovered gradient : (a) random $\mathcal{O}(h^2)$ in the tangential direction; (b)  random $\mathcal{O}(h^2)$ in the normal direction.}
	\label{fig:cex}
\end{figure}

\section{Supercloseness of geometric approximation}
\label{sec:geometric_superclose}
In this section, we introduce  the concept of geometric supercloseness between $\manifold_{h}$ and $\manifold_{h}^*$.
This will provide us the theoretical tool to show superconvergence of differential structure on deviated surfaces.
We begin with some relevant definitions on  triangular surface meshes. 
%
%
%

For any two adjacent triangles  in $\mathcal{T}_h$, we say that the two adjacent trianlges form an $O\left(h^2\right)$ approximate parallelogram \cite{BankXu2003} if the lengths of any two opposite edges differ only by $O\left(h^2\right)$. Based on the above $\mathcal{O}(h^2)$ parallelogram condition, we define the  $\mathcal{O}(h^{2\sigma})$ irregular condition for the surface meshes .

\begin{definition}
	\label{def:2sigma_irregular}
	A triangular mesh $\mathcal{T}_h$ is said to satisfy the \emph{$\mathcal{O}(h^{2\sigma})$ irregular condition} if
	there exist a partition $\mathcal{T}_{h,1} \bigcup \mathcal{T}_{h,2}$ of $\mathcal{T}_h$ and a positive constant $\sigma$ such that every two adjacent triangles in $\mathcal{T}_{h,1} $ form an $\mathcal{O}(h^{2})$ parallelogram and 
	\[\sum_{\tau_h\subset \mathcal{T}_{h,2}} \abs{\tau_h} =\mathcal{O}(h^{2\sigma}).\]
\end{definition}
A formal definition of $\mathcal{O}(h^{2})$ parallelogram can be found in \cite{DonGuo20}.
To proceed, we introduce the concept of \emph{geometric supercloseness} here.
\begin{definition}
	Let $\manifold_{h}^*$ and $\manifold_h$ be the exact interpolation and inexact approximation, respectively. We call $\manifold_h$ is \emph{geometrically superclose} to $\manifold_{h}^*$ if the following properties are satisfied:
	\begin{itemize}
		\item [(i)] Let $g_h$ and $g_h^*$ be the metric tensors associated to $\manifold_{h}$ and $\manifold_h^*$ respectively, then
		\begin{equation}\label{eq:metric_condition}
		\norm{g_h -g_h^*}_\infty = \mathcal{O}(h^2).
		\end{equation}
		\item [(ii)] The meshes of $\manifold_{h}$ and $\manifold_h^*$ have the same number of nodes and triangles. For every triangle on $\tau_{h,j}\subset \manifold_{h}$, there exists a one-to-one triangle  $\tau^*_{h,j}\subset \manifold_h^*$ correspondingly, vice versa, such that we can find local parameterizations defined on the same parametric domain: $\f r_h:\Omega \to \tau_{h,j}$ and $\f r_h^*:\Omega \to \tau^*_{h,j}$ respectively, satisfying
		\begin{equation}\label{eq:jacob_condition}
		\norm{\partial \f r_h - \partial \f r_h^* }_{\infty,\Omega}  = \mathcal{O}(h^2).
		\end{equation}
	\end{itemize}
Here and in the following, we use the notation $\partial$ to denote the Jacobian for vector valued functions, just to distinguish the gradient operator $\nabla$ for scalar functions.
\end{definition}
In fact, \eqref{eq:jacob_condition} implies \eqref{eq:metric_condition}, however, the reverse is not true. 
We show in the next section that \eqref{eq:jacob_condition} provides the ingredient for proving the superconvergence  of recovered gradient on deviated discretization of surfaces.

To make the condition more concrete, we consider the following assumptions on the triangulations.
\begin{assumption}
	\label{ass:irregular}
	\begin{itemize}
		\item[(i)] The triangulation $\manifold_h$ is shape regular and quasi-uniform. Moreover, the $\mathcal{O}(h^{2\sigma})$ irregular condition holds for  $\manifold_h$.	
		\item[(ii)] $\manifold_h$ and $\manifold_h^*$ have the same amount of triangles and vertices, and every vertex pair of $\manifold_h$ and $\manifold^*_h$ satisfy the deviation condition that 
		\begin{equation}\label{eq:close_condition1}
		\abs{x_{h,i}-x_{h,i}^*} =\mathcal{O}(h^2)  \text{ for all } \; i\in I_h.		\end{equation}
		\item [(iii)] Denote $\set{\tau_{h,j}}_{j \in J_h}\in \manifold_{h}$ and $\set{\tau_{h,j}^*}_{j \in J_h}\in \manifold_{h}^*$ the two sets of triangle meshes. For each vertex pair $(\tau_{h,j},\tau^*_{h,j})_{j \in J_h}$  when we do a parallel shift of the two set of meshes and move one vertex pair to a common point, then the new pairs of vertices (denoted by $(\xi_{k_j,h},\xi_{k_j,h}^*)_{j \in J_h}$ after the position transformation) satisfy the distance condition that
		\begin{align}\label{eq:close_condition2}
	 &\abs{P_\tau(\xi_{k_j,h}- \xi_{k_j,h}^*)} =\mathcal{O}(h^3)	\; \\
  \text{ and }\;& \abs{P_n(\xi_{k_j,h}- \xi_{k_j,h}^*)} =\mathcal{O}(h^3),\; \text{ for all } k_j \text{ and } j \in J_h,\label{eq:close_condition3}
		\end{align}
		where $P_\tau$ and $P_n$ are the tangential and normal projections to $\manifold_h^*$ respectively. Alternatively, one can consider the normal and tangential projections to $\manifold_h$ as well. 
	\end{itemize}
\end{assumption}
We emphasis that the condition in Assumption \ref{ass:irregular} might not be the only case that leads to geometric supercloseness. However, it is quite practical to be fulfilled by the surface discretization algorithms using first order projections.

The first condition of Assumption \ref{ass:irregular} is quite standard, and it is crucial for proving superconvergence in the literature, e.g., \cite{BankXu2003,XuZhang2004,WeiChenHuang2010}.
However, in the surface setting, this condition has been assumed on $\manifold_h^*$, which is the exact interpolation of $\manifold$.
Note that condition \eqref{eq:close_condition1} is exactly \eqref{eq:vertex_error}. In  addition, we show that \eqref{eq:close_condition2} and   \eqref{eq:close_condition3} are also required for establishing the superconvergence on deviate surfaces.

In the following, we show that Assumption \ref{ass:irregular} implies the \textsf{geometric supercloseness} for $\manifold_{h}$ to approximate $\manifold_{h}^*$. To see that, we show some auxiliary results first.

We consider $\tau_h$ and $\tau^*_h$ to be an triangle pair taken from $\manifold_{h}$ and $\manifold_{h}^*$. 
In the following analysis, we take one of triangle as a parametric domain, e.g., $\tau_h$, then there exist an linear map $\Gamma$ such that $ \tau^*_h= \Gamma( \tau_h )$.  Note that $\Gamma: \R^2 \to \R^3$, therefore $\partial \Gamma$ is a $3\times 2$ constant matrix, which is invariant with respect to parallel shift of $ \tau^*_h$.
We denote $\operatorname{Id}=\left( \begin{matrix}
1 & 0 & 0\\
0 & 1 & 0
\end{matrix}\right)^\top$.

We start with establishing  the relationship  between  the edges of the pair of triangles. 
\begin{lemma}\label{lem:aux_0}
	Let $\tau_h$ and $\tau^*_h$ be shape regular triangle pair of diameter $h$ from $\manifold^*_{h}$ and $\manifold_{h}$ respectively, and the distance of each of their vertex pair satisfies condition \eqref{eq:close_condition1}, for $h  \leq 1$ sufficiently small. In addition, if the bound in \eqref{eq:close_condition2} or \eqref{eq:close_condition3}  hold, i.e., either the one with tangential or the one with normal projection.
	Then the following error bounds hold 
	\begin{equation}\label{eq:edge_diff}
	\abs{l_{k}^*-l_{k}}  =\mathcal{O} (h^3)\;   \text{ and }  \;  \abs{(l^*_{k})^2-(l_{k})^2}=  \mathcal{O} (h^4)  \; \text{ for all }\;  k=1,2,3 ,
	\end{equation}
	where $\set{l_{k} | {k=1,2,3}}$ (or $\set{l_{k}^* | {k=1,2,3}}$) denotes the length of the three edges $\set{e_{k} | {k=1,2,3}}$ (or $\set{e_{k}^* | {k=1,2,3}}$ ) of triangle $\tau_h$ (or $\tau^*_h$).
\end{lemma}
\begin{proof}
 To show the first inequality,  we do parallel and vertical translation of $e_{k}^*$ to a common point with $e_{k}$. Then,  for sufficiently small $h$,  we have 
	\begin{equation*}
	\begin{aligned}
	l_{k}^*  \leq & \sqrt{ (l_{k}+ c_{2,k}h^3 )^2+c_{1,k}h^4} \\
	=& l_{k}\sqrt{ 1+ 2c_{2,k}\frac{h^3}{l_k} + c_{2,k}^2 \frac{h^6}{l^2_{k}}+c_{1,k}\frac{h^4}{ l^2_{k}}}\\
	\leq & l_{k} \left(1 + C (\frac{h^3}{l_k} +  \frac{h^6}{l^2_{k}}+ \frac{h^4}{ l^2_{k}} ) \right).
	\end{aligned}
	\end{equation*}
	Here both $c_{1,k}$ and $c_{2,k}$ are either positive or negative constants, which are corresponding to the tangential and normal decomposed distance mismatches respectively. 
	Noticing  the fact that $l_{k}\sim h$ for all $k=1,2,3$,  we obtain  the first estimate in \eqref{eq:edge_diff}.

	For the second estimate in \eqref{eq:edge_diff}, it is sufficient to use $a^2-b^2=(a-b)(a+b)$, take into account that the edge lengths are of order $h$, and combine with the first estimate we have the conclusion.
\end{proof}
 
Using the above relationship of edges between the two triangles, we can prove the following result.
\begin{lemma}\label{lem:aux1}
	Under the same condition as Lemma \ref{lem:aux_0},  we have
	\begin{equation}\label{eq:triangle_area}
	\abs{\mathcal{A}(\tau^*_h)  - \mathcal{A}(\tau_h) } = \mathcal{O}( h^4)  \; \text{ and }  \;  \abs{\frac{\mathcal{A}(\tau_h^*) }{\mathcal{A}(\tau_h) } -1} = \mathcal{O}( h^2),
	\end{equation}
	where $\mathcal{A}$ is the area function.
\end{lemma}
\begin{proof}
Lemma \ref{lem:aux_0} implies there exists $c_k$  either being positive or negative such that 
 $l_{k}^*=l_{k}+c_k h^3$.  Using the Heron's formula,  we  can calculate the area of $\tau_h$ and $\tau_h^*$ as
	\[\mathcal{A}(\tau_h)=\sqrt{s_1\prod_k(s_1-l_{k})}  \text{ and }  \mathcal{A}(\tau_h^*)=\sqrt{(s_1+d_0(\f c) h^3)\prod_k(s_1-l_{k} + d_k(\f c) h^3)} , \]
	where $s_1=\sum_k \frac{l_{k}}{2}$, and $\f c=(c_1,c_2,c_3)$.  Thus, we have
	\begin{equation}\label{eq:area_est}
	\begin{aligned}
	\abs{ \frac{\mathcal{A}(\tau_h^*)}{\mathcal{A}(\tau_h)} -1}
	= &\abs{\sqrt{(1+d_0(\f c) \frac{h^3}{s_1})\prod_k(1 + d_k(\f c)\frac{ h^3}{s_1-l_{k} })}-1} \\
	= & \abs{\sqrt{1+q_0(\f c)h^2} \prod_k \sqrt{1 + q_k(\f c)h^2} -1}\\
	\leq &  \abs{(1+Cq_0(\f c)h^2) \prod_k (1 + Cq_k(\f c)h^2) -1}\\
	\leq & C(\f q) h^2
	\end{aligned}
	\end{equation}
	for some constants $\f q=(q_0,q_1,q_2,q_3)$ depending on $\f c$, which gives a constant $C$ dependent of $\f c$.
	The first estimate in \eqref{eq:triangle_area} is obvious since $\mathcal{A}(\tau_h)$ is also of order $h^2$. Multiply with $\mathcal{A}(\tau_h)$ on both side of \eqref{eq:area_est} gives the estimate.
\end{proof}

Based on the above two lemmas, we can show the following approximation results for Jacobian matrix and metric tensor.
\begin{proposition}\label{prop:aux2}  Assume the same condition as Lemma \ref{lem:aux_0}. Let $\Gamma: \tau_h \to \tau_h^*$ be the linear transformation, and let $g_\Gamma := (\partial \Gamma)^\top\partial \Gamma$. 
	Then we have the following relations:
	\begin{itemize}
		\item [(i)] The determinate of $g_\Gamma$ satisfies
		\begin{equation}\label{eq:aux_equiv2}
		\abs{\sqrt{\det{g_\Gamma} }- 1} = \mathcal{O}(h^2).
		\end{equation}
		\item [(ii)] The Jacobian $\partial \Gamma$ and the metric $g_\Gamma$  have the following estimate
		\begin{equation}\label{eq:aux_equiv1}
		\norm{\partial \Gamma-  \operatorname{Id}}_\infty= \mathcal{O}(h) \; \text{ and } \;	\norm{g_\Gamma -\operatorname{I} }_\infty = \mathcal{O}(h^2),
		\end{equation}
		respectively, where $\operatorname{I}= \operatorname{Id}^\top \operatorname{Id}$ is the $2\times 2$ identity matrix.
		\item [(iii)] Let both conditions in \eqref{eq:close_condition2} and  \eqref{eq:close_condition3}  hold. Then we have an improved error estimate for the Jacobian matrix
		\begin{equation}\label{eq:aux_equiv0}
		\norm{\partial \Gamma-  \operatorname{Id}}_\infty= \mathcal{O}(h^2).
		\end{equation}
	\end{itemize}
\end{proposition}
\begin{proof}	
	(i) To show \eqref{eq:aux_equiv2}, notice $\mathcal{A}(\tau_h^*) =\sqrt{\det{g_\Gamma}} \mathcal{A}(\tau_h) $.
	The second estimate of \eqref{eq:triangle_area} in Lemma \ref{lem:aux1} implies
	\[\abs{\sqrt{\det{g_\Gamma}} -1} =\abs{ \frac{\mathcal{A}(\tau_h^*) }{\mathcal{A}(\tau_h) }-1}  =\mathcal{O}(h^2) \]
	which concludes the proof of \eqref{eq:aux_equiv2}.
	
	(ii) For the first one in \eqref{eq:aux_equiv1}, we notice that 
	$g=(\partial \Gamma)^\top \partial \Gamma $. Since we choose $\tau_h$ to be the parametric domain of $\tau_h^*$, it follows that  $\tau_h^* =\partial \Gamma \tau_h + x_0$. Now we transfer the two triangles to a common vertex, and change the Cartesian coordinate and let the common vertex be at the origin.
	Let the position of two other vertices of $\tau_h$ in the new coordinate be $\xi_1=(\xi_{1,1},\xi_{1,2},0)$ and $\xi_2=(\xi_{2,1},\xi_{2,2},0)$, and the corresponding vertices of $\tau_h^*$ be $\psi_1= (\psi_{1,1},\psi_{1,2},\psi_{1,3})$ and $\psi_2= (\psi_{2,1},\psi_{2,2},\psi_{2,3})$.
	Then we have $\psi_k=\partial \Gamma \xi_k$ for $k=1,2$, which gives
	\[ \partial \Gamma = 
	\left(\begin{matrix} 
	\psi_{1,1}& \psi_{2,1}\\
	\psi_{1,2}& \psi_{2,2}\\
	\psi_{1,3}& \psi_{2,3}\\
	\end{matrix}\right)
	\left(\begin{matrix} 
	\xi_{1,1}& \xi_{2,1}\\
	\xi_{1,2}& \xi_{2,2}\\
	\end{matrix}\right)^{-1}.
	\]
It is not hard to see that 
	\begin{equation}\label{eq:est_Gamma1}
	\partial \Gamma -  \operatorname{Id}=\left( \left(\begin{matrix} 
	\psi_{1,1}& \psi_{2,1}\\
	\psi_{1,2}& \psi_{2,2}\\
	\psi_{1,3}& \psi_{2,3}\\
	\end{matrix}\right) - \left(\begin{matrix} 
	\xi_{1,1}& \xi_{2,1}\\
	\xi_{1,2}& \xi_{2,2}\\
	0 & 0
	\end{matrix}\right) \right)
	\left(\begin{matrix} 
	\xi_{1,1}& \xi_{2,1}\\
	\xi_{1,2}& \xi_{2,2}\\
	\end{matrix}\right)^{-1} .
	\end{equation}
	Because $(\xi_{i,j})_{i,j\in \set{1,2}}$ are all of order $h$, and $ (\abs{\xi_{1,1}\xi_{2,2}- \xi_{2,1}\xi_{1,2}})/2$ equals the area of $\tau_h \approx \mathcal{O}(h^2)$,     we can conclude that 
	\begin{equation}\label{eq:est_Gamma2}
	\left(\begin{matrix} 
	\xi_{1,1}& \xi_{2,1}\\
	\xi_{1,2}& \xi_{2,2}\\
	\end{matrix}\right)^{-1} = \left(\begin{matrix} 
	\xi_{2,2}& -\xi_{1,2}\\
	-\xi_{2,1}& \xi_{1,1}\\
	\end{matrix}\right)  (\abs{\xi_{1,1}\xi_{2,2}- \xi_{2,1}\xi_{1,2}})^{-1} = M/h, 
	\end{equation}
	where $M$ is some constant matrix independent of $h$.
	
	Recall that $\psi_1 $, $\psi_2$ are vertices of $\tau_h^*$, and $\xi_1$, $\xi_2$ are vertices of $\tau_h$ in the new coordinates. According to  the assumption \eqref{eq:close_condition1}, 
	we have 
	\begin{equation}\label{eq:h3_tangent}
	\abs{\psi_{i,j} -\xi_{i,j}}=\mathcal{O}(h^3) \; \text{ for } \quad i=1,2,\;\; j=1,2,
	\end{equation}
	and the normal part 
	\begin{equation}\label{eq:h2_normal}
	\abs{\psi_{i,3}}=\mathcal{O}(h^2) \text{ for } i=1,2.
	\end{equation}
Plugging  the above estimate  into \eqref{eq:est_Gamma1}, we obtain
	\[ \norm{\partial \Gamma -\operatorname{Id}}_\infty =\mathcal{O}(h).  \]
	Now we show the second inequality  in \eqref{eq:aux_equiv1}. Since $g_\Gamma = (\partial \Gamma)^{\top}\partial \Gamma$, it follows that
	\begin{equation}\label{eq:est_g}
	g_\Gamma -  \left(\begin{matrix} 
	1& 0\\
	0& 1\\
	\end{matrix}\right)= \left(\begin{matrix} 
	\xi_{1,1}& \xi_{1,2}\\
	\xi_{2,1}& \xi_{2,2}\\
	\end{matrix}\right)^{-T}
	\left( \left(\begin{matrix} 
	\psi_1\cdot \psi_1 & \psi_1\cdot \psi_2\\
	\psi_2\cdot \psi_1 & \psi_2\cdot \psi_2
	\end{matrix}\right) - \left(\begin{matrix} 
	\xi_1\cdot \xi_1 & \xi_1\cdot \xi_2\\
	\xi_2\cdot \xi_1 & \xi_2\cdot \xi_2
	\end{matrix}\right) \right)
	\left(\begin{matrix} 
	\xi_{1,1}& \xi_{2,1}\\
	\xi_{1,2}& \xi_{2,2}\\
	\end{matrix}\right)^{-1} .
	\end{equation}
	Note that $\psi_i\cdot \psi_i$ for $i=1,2$ are the squared length of the edges of $\tau_h^*$ and  $\xi_i\cdot \xi_i $ are squared length of the edges of $\tau_h$.
	Then, Lemma \ref{lem:aux_0} implies 
	\[\abs{\psi_i\cdot \psi_i - \xi_i\cdot \xi_i }  =\mathcal{O}(h^4), \text{ for } i=1,2.\]
For the mixed term, the relationship \eqref{eq:h3_tangent} and the triangle inequality implies that 
\begin{align*}
	\abs{\psi_i\cdot \psi_j - \xi_i\cdot \xi_j }
	=  \abs{\psi_i\cdot \psi_j - \psi_i \cdot \xi_j +   \psi_i \cdot \xi_j - \xi_i\cdot \xi_j }
	\le  \abs{\psi_i}\abs{ \psi_j -  \xi_j} +  \abs{ \psi_i  - \xi_i }\abs{\xi_j} 
	=\mathcal{O}(h^4),
\end{align*}
where we have used  the fact that the length of the edges are of $\mathcal{O}(h)$ in the last inequality.  Substituting the above estimates into \eqref{eq:est_g} completes the proof of the second result in \eqref{eq:aux_equiv1}.

	(iii) For \eqref{eq:aux_equiv0}, 
	notice that the additional condition in \eqref{eq:close_condition2} of normal direction error increases the order of $\abs{ \psi_{i,3}}$ in \eqref{eq:h2_normal}. That is
	\[ \abs{ \psi_{i,3}}=\mathcal{O}(h^3)  \quad \text{ for } i=1,2,\]
	which is  of the same order as $ \abs{ \psi_{i,j}-\xi_{i,j}}$ for $i,j=1,2$ in \eqref{eq:h3_tangent}.
	Thus we conclude the statement by comparing the proof of the first estimate in \eqref{eq:aux_equiv1}.
\end{proof}
Note that with only assumption \eqref{eq:close_condition1} and either the tangential condition in \eqref{eq:close_condition2} or the one in \eqref{eq:close_condition3} , it is already able to prove the metric condition \eqref{eq:aux_equiv2}. However, it is not sufficient to conclude the Jacobian condition \eqref{eq:aux_equiv0}, which is one of the essential ingredients  for the geometric supercloseness.

With the above preparation, we are ready to show that Assumption \ref{ass:irregular} gives us supercloseness of the geometric approximation.
\begin{proposition}
	\label{prop:geometry}
	Let $\manifold_h$  and $\manifold^*_h$ satisfy Assumption \ref{ass:irregular}, then we have
	\begin{itemize}
		\item[(i)] The triangulation $\manifold^*_h$ is also shape regular and quasi-uniform, and the $\mathcal{O}(h^{2\sigma})$ irregular condition is fulfilled for $\manifold^*_h$.
		\item[(ii)] The local piecewise linear parametrization functions, $\f r_{h,i} $ and $\f r^*_{h,i}$, satisfy
		\begin{equation}
		\label{eq:gs_condition}
		\norm{\partial \f r^*_{h,i}-\partial \f r_{h,i}}_{\infty,\Omega_i} = \sqrt{\mathcal{A}(\manifold_{h,i}) }\mathcal{O}(h^{2}),\;\; \text{ for all } \; i\in I_h. 
		\end{equation}
		Here   $\mathcal{A}$ is the area functional, and $\manifold_{h,i}\subset \manifold_h$ is the patch corresponding to the parameter domain $\Omega_i$. To simplify the discussion, we choose the same $\Omega_i $ for both $\manifold_{h,i}$ and $\manifold_{h,i}^*$, which is obtained as in Algorithm \ref{alg:isop_gradient}.
	\end{itemize}
\end{proposition}
\begin{proof}
	The first assertion holds because the triangles of $\manifold_{h}$ are shape regular and quasi-uniform, and they satisfy the $\mathcal{O}(h^{2\sigma})$ irregular condition above.
	The definition of shape regular and quasi-uniform can be found in many textbooks of finite element methods \cite{BS2008, Ci2002}.
	Using the triangle inequality, we get the conclusion.
	
	Now we proceed to  prove the second assertion.
	Let $\Omega_i$ be the parameter domain for patches selected around the vertex $x_{h,i}$. Denote the local index-set associated to vertices of the selected patch around $x_{h,i}$ to be $J_i$. We notice that both $\f r_{h,i}$ and $\f r^*_{h,i}$ are piecewise linear functions defined on $\Omega_i$. 
	Let $\Omega_{i,j}$ be the common parameter domain for the corresponding triangle pairs $\tau_{h,j}$ and $\tau^*_{h,j}$ for the index $j\in J_i$.
	Then  $\f r_{h,i}$ and $\f r^*_{h,i}$ are affine functions on each of the triangle regions $\Omega_{i,j}$ for every $j\in J_i$. Note that due to the construction of $\Omega_i$ (see Algorithm \ref{alg:isop_gradient}), the change of coordinates from $\Omega_{i,j}$ to $\tau_{h,j}$ or $\tau^*_{h,j}$ is done by affine transformation.
	
	Since the Jacobian $\partial \f r_{h,i}$ and $\partial \f r^*_{h,i}$ are constant functions,  we have
	\begin{equation*}
	\begin{aligned}
	\norm{\partial \f r^*_{h,i} -\partial \f r_{h,i}}_{\infty,\Omega_{i,j}}
	& \leq   \left(\sup_{i,j}{\sqrt{\frac{\mathcal{A}(\Omega_{i,j})}{\mathcal{A}(\tau_{h,j})}}}\right) \norm{\partial R_{i,j}} \norm{\partial \Gamma_{j} -\operatorname{Id}}_{\infty,\tau^*_{h,j}}\\
	& \leq \norm{\partial \Gamma_{j} -\operatorname{Id}}_{\infty,\tau^*_{h,j}}.
	\end{aligned}
	\end{equation*} 
	Here $\partial R_{i,j}$ is a $3\times 3$ matrix to change coordinates from $\tau_{h,j}$ to $\Omega_{i,j}$, which is a unitary matrix, and thus $\norm{\partial R_{i,j}}=1$; $\Gamma_{j} $ is the geometric mapping lift $\tau_{h,j}$ to $\tau^*_{h,j}$, which can be obtained from $\f r_{h,i}$ by changing the local coordinates from $\Omega_{i,j}$ to $\tau_{h,j}$, while the mapping $\f r_{h,i}$ becomes identity after the change of coordinates.
	Due to the shape regularity of $\manifold_h$ and $\manifold^*_h$, and $\manifold$ has bounded curvature, thus all the elements of $\set{\sqrt{\frac{\mathcal{A}(\Omega_{i,j})}{\mathcal{A}(\tau_{h,j})}}}_{i\in I_h, j\in J_h}$ are uniformly bounded from below and from above on all the parametric domain $\set{\Omega_i}_{i\in I_h}$ and the triangles $\set{\tau_{h,j}}_{j\in J_h}$. Typically, we have $\sup_{i,j}\sqrt{\frac{\mathcal{A}(\Omega_{i,j})}{\mathcal{A}(\tau_{h,j})}}\leq 1$ due to the fact that $\Omega_{i,j}$ is the projection of $\tau_{h,j}$ onto $\Omega_i$.
	
	Applying  the result from Proposition \ref{prop:aux2}, we have
	\[\norm{\partial \f r_{h,i} - \partial \f r^*_{h,i} }_{\infty,\Omega_{i,j}}\leq \norm{\partial \Gamma_{j} - \operatorname{Id}}_{\infty,\tau_{h,j}} \leq C\sqrt{\mathcal{A}(\tau_{h,j}) }h^2. \]
	Summing up over $j\in J_i$,  we arrive at the following inequality
	\begin{equation*}
	\norm{\partial \f r_{h,i} - \partial \f r^*_{h,i} }_{\infty,\Omega_{i}}^2
	=\mathcal{O}(h^4) \sum_{j\in J_i}\mathcal{A}(\tau_{h,j}) .
	\end{equation*}
	Taking square root on both sides gives the conclusion of the second assertion.
\end{proof}

The relation in \eqref{eq:gs_condition} tells some regularity on the approximations of $\manifold_h$ to $\manifold$.
It is similar to the supercloseness property for the gradient of the finite element solutions to the gradient of the interpolation of the exact solutions \cite{BankXu2003,XuZhang2004}.
This  leads to the superconvergence of the differential structure on deviated surfaces.
In what follows, we present why the superconvergence of the differential structure is crucial for the superconvergence of the gradient recovery on perturbed surfaces.

\section{Superconvergent gradient recovery on deviated surfaces}
\label{sec:app_1}
\subsection{Parametric gradient recovery schemes on surfaces}
Here, we generalize the idea of parametric polynomial preserving recovery proposed in \cite{DonGuo20} to have a general family of recovery methods in surfaces setting. More precisely, the algorithmic
 framework  generalizes  the methods introduced in \cite{WeiChenHuang2010}, which ask for exact geometry-prior, to the case without  exact geometry-prior.
To do this, we use the intrinsic definition of gradient operator on surfaces. Given a local parametric patch, and $\f r:\Omega \to \manifold$ the parametrization function of this patch, define $\bar{u}:=u\circ \f r$. Then we have
\begin{equation}
\label{eq:local_gradient}
(\nabla_g u)\circ \f r=\nabla \bar{u}  (g\circ \f r)^{-1} \partial \f r  = \nabla \bar{u}( \partial \f r )^\dag \quad \text{ on } \Omega.
\end{equation}
Here $( \partial \f r )^\dag$ is the pseudo-inverse of $\partial \f r$ the Jacobian of the geometric mapping over $\Omega$.
For a small digest of differential operator on surfaces, we refer to the appendix of \cite{DonJueSchTak17} and also the background part in \cite{DonGuo20}. One may refer to the textbooks, e.g., \cite{doCarmo1992,Lee12} for more comprehensive introduction on differential geometry.
Inspired from  \eqref{eq:local_gradient}, we try to recover the gradient on surfaces using a two-level strategy. That is to recover the Jacobian $\partial \f r$ and also $\nabla \bar{u}$ iso-parametrically on every local patch.
We call this family of methods the isoparametric gradient recovery schemes.
We shall see that they ask for neither the exact vertices nor the precise tangent spaces.
In particular the PPPR method which was introduced in \cite{DonGuo20} can also be included into this framework.

\begin{algorithm}\label{alg:isop_gradient}
	\caption{Isoparametric gradient recovery schemes}
	\label{alg:isop_gradient}
	Input: Discretized triangular surface $\manifold_h$ with vertices set $(x_{h,i})_{i\in I_h}$ and the data (FEM solutions) $(u_{h,i})_{i\in I_h}$.
	Then we suggest two-stage recovery steps for all $i\in I_h$.
	
	$\bullet$ Geometric recovery:
	\begin{itemize}
		\item[(1)] At every $x_{h,i}$, select a neighboring patch $\manifold_{h,i}\in \manifold_h$ around $x_{h,i}$ with sufficient vertices. Compute the unit normal vectors of every triangle faces in $\manifold_{h,i}$. Compute the simple (weighted) averaging of the unit normal vectors, and normalize it to be $\phi^{3}_i$. 
		Take the orthogonal space to $\phi^{3}_i$ to be the parametric domain $\Omega_i$. Shift $x_{h,i}$ to be the origin of $\Omega_i$, and choose $(\phi^1_i,\phi^2_i)$ the orthonormal basis of $\Omega_i$. 
		
		\item[(2)] Project all selected vertices of $\manifold_{h,i}$ around $x_{h,i}$ onto the parametric domain $\Omega_i$ from Step $(1)$, and record the new coordinates as $\zeta_{i_j}$. 
		
		\item[(3)] Use a planar recovery scheme $R^k_h$ to recover the surface Jacobian with respect to $\Omega_i$.
		Typically, this is considered every surface patch as local graph of some function $s$, that is $\f r_i=(\Omega_i,s_i(\Omega_i))$. Then the recovered Jacobian at the selected patch is $J_{r,i}=(\operatorname{I}, R^k_hs_{i,j})^\top$, where $\operatorname{I}$ is the identity matrix according to the dimension of $\Omega_i$.
	\end{itemize}	
	
		$\bullet$ Function gradient recovery:
	\begin{itemize}	
		\item[(4)] Let $(\bar{u}_{h,j})_{j\in I_i}$ be the values of $(u_{h,j})_{j\in I_i}$ associated to vertices of $\manifold_{h,i}$ isoparametrically defined on the parameter domain $\Omega_i$. Then using the same planar recovery scheme $R^k_h$ to recover gradient from $(\bar{u}_{h,j})_{j\in I_i}$  with respect to parameter domain $\Omega_i$.
		
		\item[(5)] In the spirit of \eqref{eq:local_gradient}, use the results from Step $(3)$ and Step $(4)$ to get the recovered surface gradient at $x_{h,i}$: 
		\begin{equation}
		\label{eq:gradient_pgr}
		G^k_h u_{h,i}  = R^k_h \bar{u}_{h,i}
		(J_{r,i})^\dag(\phi^1_i, \phi^{2}_i, \phi^{3}_i),
		\end{equation}
		where $(J_{r,i})^\dag=(J_{r,i} J_{r,i}^\top)^{-1}J_{r,i}$. 
		The orthonormal basis $\set{\phi^1_i, \phi^{2}_i, \phi^{3}_i}$ is multiplied to unify the coordinates from local ones to a global one in the ambient Euclidean space.
	\end{itemize}
	Output: The recovered gradient at selected nodes $\set{G^k_h u_{h,i} }_{i\in I_h}$. For $x$ being not a vertex of triangles, we use linear finite element basis to interpolate the values $\set{G_{h}  u_{h,i}}_{i\in I_h}$ at vertices of each triangle.
\end{algorithm}

Algorithm \ref{alg:isop_gradient} consists of two main parts, and describes a family of recovery methods as the planar recovery scheme can be varied. The generality should also cover higher dimensional problems, but for simplicity, we focus on $2$-dimensional case only.
For selecting a concrete planar recovery method in both Step $(3)$ and Step $(4)$, actually, almost all the local recovery methods for functions in the Euclidean domain can be applied. 
This will include many of the methods which have been discussed in \cite{WeiChenHuang2010}. Though the preferred candidates will be $ZZ$-scheme and PPR. The latter gives then the PPPR method proposed in \cite{DonGuo20}.
In view of \eqref{eq:gradient_pgr}, one can see clearly that the proposed scheme is an approximation of \eqref{eq:local_gradient} at every nodes:  $R^k_h \bar{u}_{h,i}$ recovers $\nabla \bar{u}$, and  $ (J_{r,i})^\dag$ recovers $( \partial \f r )^\dag$.
Moreover, it gives the intuition that why the result in \eqref{eq:gs_condition} is required, in order to match the superconvergence of the function gradient recovery and the superconvergence of the geometric structures simultaneously.
Next, we prove the superconvergence property of the recovery scheme.

\subsection{Superconvergence analysis on deviated geometry}
Even though a general algorithmic framework is described, which can cover several different methods under the same umbrella,  we will focus on parametric polynomial preserving recovery (PPPR) scheme for the theoretical analysis. 
The reason is that PPPR asks for minimum requirements on the meshes in comparison with several other methods. For instance, the simple (weighted) average method, or the generalized ZZ-scheme, both of which require an additional $\mathcal{O}(h^2)$-symmetric condition on the meshes.
However, the general idea of the proof should be able to cover these methods giving the necessary mesh conditions.

We take the following Laplace-Beltrami equation  as our model problem to conduct the analysis:
\begin{equation}
\label{eq:laplace}
-\lapla_g u= f \quad \text{ where } \; \int_{\manifold} u \; dvol =0.
\end{equation}
The weak formulation of equation \eqref{eq:laplace} is given as follows: Find $u\in H^1(\manifold)$ with $\int_{\manifold} u \; dvol =0$  such that
\begin{equation}
\label{eq:var_laplace}
\int_{\manifold} \nabla_g u \cdot \nabla_g v \; dvol=\int_{\manifold} f v \; dvol, \quad  \forall v\in H^1(\manifold).
\end{equation}
The regularity of the solution has been proven in \cite[Chapter 4]{Aubin1982}. In the surface finite element method, the surface $\manifold$ is approximated by the triangulation $\manifold_h$ which satisfy Assumption \ref{ass:irregular},
and the solution is simulated in the continuous piecewise linear function space $\mathcal{V}_h$ defined over $\manifold_h$, i.e. 
\begin{equation}
\label{eq:d_var_laplace}
\int_{\manifold_h} \nabla_{g_h} u_h \cdot \nabla_{g_h}  v_h \; dvol_h=\int_{\manifold_h} f_h v_h \; dvol_h, \quad  \forall v_h\in \mathcal{V}_h(\manifold_h).
\end{equation}

We first show that there exists an underlying smooth surface denoted by $\widetilde{\manifold}_h$ so that $\manifold_h$ can be thought as an interpolation of it.
This intermediate surface is not needed practically in the algorithm, but it is helpful for our error analysis.
For such purpose, we recall a fundamental result in differential topology attributed to Whitney (see, e.g., \cite[Theorem 6.21]{Lee12}).
\begin{theorem}[Whitney Approximation Theorem]\label{thm:Whitney} 
Suppose $\mathcal{M}$ is a smooth surface with or without boundary, and $F:\mathcal{M}\to \R^{d+1}$ is a continuous
map. Given any positive continuous function $\epsilon:\mathcal{M}\to \R^+$, there exists a smooth
function $F_\epsilon:\mathcal{M}\to \R^{d+1}$ that is $\epsilon$-close to $F$. If $F$ is smooth on a closed subset $A\subset \mathcal{M}$, then $F_\epsilon$ can be chosen to be equal to $F$ on $A$.
\end{theorem}
Here $\epsilon$-close means $\abs{F(p)-F_\epsilon(p)}\leq \epsilon(p)$  for all $p\in \mathcal{M}$. Since $\epsilon$ is any positive function, and it can be arbitrarily close to $0$. $F$ and $F_\epsilon$ can be considered as mappings between $\mathcal{M}$ and other geometric objectives embedded in $\R^{d+1}$.

\begin{proposition}
	\label{prop:appr_surfaces}
	Let $\manifold$ be the precise surface, and $\manifold_h$ and $\manifold_h^*$ satisfy the assumption \ref{ass:irregular}.
	Then the following statements hold true:
	\begin{itemize}
		\item[(i)]	There exists a $C^3$ (in fact $C^\infty$) smooth surface $\widetilde{\manifold}_h$, so that $\manifold_h$ is a linear interpolation of $\widetilde{\manifold}_h$ at the vertices. 
		\item[(ii)]  Let $\f r_{\tau_{h,j}}$ and $\tilde{\f r}_{\tau_{h,j}}$ be the parametrization of the curved triangular surfaces $\tau_{j}\subset \manifold $ and $\tilde{\tau}_{h,j} \subset \widetilde{\manifold}_h$ from the triangle $\tau_{h,j}$, respectively. Then there is the estimate
		\begin{equation}\label{eq:geometric_error2}
		\norm{ \partial \tilde{\f r}_{\tau_{h,j}} - \partial \f r_{\tau_{h,j}} }_{\infty,\tau_{h,j}}  \leq C h^2 ,
		\end{equation}
		where $C$ is a constant independent of $h$.
		\item[(iii)] Let $v:\manifold \to \R$ be functions in $W^{k,p}(\manifold)$,
		and $\tilde{v}_h$ be the pullback of $v$ to $\widetilde{\manifold}_h$, then we have 
		\begin{equation}\label{eq:norm_equ}
		C_1 \norm{\tilde{v}_h}_{W^{k,p}(\widetilde{\manifold}_h)} \leq \norm{v}_{W^{k,p}(\manifold)} \leq C_2 \norm{\tilde{v}_h}_{W^{k,p}(\widetilde{\manifold}_h)},
		\end{equation}
		for some constants $0<C_1 \leq C_2$.
	\end{itemize}
\end{proposition}
\begin{proof}
	$(i)$ For the first statement, we construct in the following the smooth surface $\widetilde{\manifold}_h$ which is needed for the error analysis.
	\begin{itemize}
 \item  We construct a piece-wise cubic polynomial patch on $\manifold^*_h$. These polynomial patches are able to preserve the surface Jacobian at the vertices over the parametric triangles $\set{\tau^*_{h,j}}_{j\in J_h}$. To have this, we subdivide $\manifold$ by $\set{\tau_j}_{j\in J_h}$ which are parametrized by $\set{\tau^*_{h,j}}_{j\in J_h}$, i.e. $\tau_j=\f r_{\tau_{h,j}}^*(\tau^*_{h,j})$, where $\f r_{\tau_{h,j}}^*=(\gamma_1,\gamma_2,\gamma_3)^\top$.
 Then we reconstruct a cubic polynomial by using the value of $\gamma_k$ and the directional derivatives of $\gamma_k$ along the edges of $\tau^*_{h,j}$ for $k=1,2,3$ at the three vertices, and also the value of $\gamma_k$ at the barycenter of $\tau^*_{h,j}$. With these 10 linear independent equations for each $\gamma_k$, we reconstruct three cubic polynomial functions $\gamma_{h,k}$ which approximate $\gamma_k$ on every $\tau^*_{h,j}$, respectively. This reconstruction is done for every $j\in J_h$, then we have a closed surface which consists of piece-wise cubic polynomial patches, we denote it to be $\overline{\manifold}^*_h$. The closeness can be judged by the uniqueness of the cubic polynomial functions given the same interpolation condition (after unifying the coordinates) over every edge of $\tau^*_{h,j}$.
\item  We turn to $\manifold_h$ for a similar reconstruction. This can be done easily now.
  For every triangle $\tau_{h,j}$ with $j\in J_h$, we define an affine mapping $\Gamma_j: \tau_{h,j} \to \tau^*_{h,j}$. Note that polynomials composed with affine mappings are still polynomials. Therefore we can have the polynomial patches over $\tau_{h,j}$ by composing the previously reconstructed polynomial function on every $\tau^*_{h,j}$ with the corresponding $\Gamma_j$ function. This then gives another piece-wise cubic polynomial patches which we denote by $\overline{\manifold}_h$. It is again closed due to the interpolation property of the affine mapping on every edge  $\tau_{h,j}$, that is a unique affine map  which results a unique cubic polynomial with the composition on every edge. Therefore $\overline{\manifold}_h$ is also a closed piece-wise cubic polynomial surface.
	\end{itemize}
The above construction of $\overline{\manifold}_h$ shows that it is a $C^0$ surface.
Now we apply Whitney Approximation Theorem \ref{thm:Whitney} to have a smooth approximation of $\overline{\manifold}_h$. For this we define the map $F:\manifold \to \overline{\manifold}_h\subset \R^3$ which is a $C^0$ isomorphism, i.e. it gives a triangulation of $\manifold$ such that every triangle on $\overline{\manifold}_h$ has a unique triangle on $\manifold$ as the pre-image of $F$. All these triangles are non-intersect and give a partition of $\manifold$. Then there exists compatible $C^k$ (for any $k\in [1,\infty]$) smooth approximation which can be arbitrarily close to $F$, i.e., for all $\epsilon>0$ there exists a family of $C^\infty$ isomorphism $F_\epsilon:\manifold \to  \widetilde{\manifold}^\epsilon_h$ which approximates $F$. For every fixed $h$, we choose a open set denoted by $E_\delta\subset \manifold$ which consists of the narrow band in the neighbourhood of all the edges of the curved triangles on $\manifold$. Here $\delta$ denotes the width of every band, and note that $\delta$ can be arbitrarily small. Since $\overline{\manifold}_h$ is buildup by piece-wise polynomial patches, $F$ is $C^\infty$ smooth in $\manifold \backslash E_\delta$.  Then from Theorem \ref{thm:Whitney} we have that $F_\epsilon=F$ over the set $\manifold \backslash E_\delta$. In fact we can choose $\delta=\epsilon$.
As $\epsilon$ can be arbitrarily close to $0$, and we use the notation $\widetilde{\manifold}_h$ for the smoothed version of $\overline{\manifold}_h$.
Note that when $\epsilon$ approaching $0$, all the vertices on $\manifold_{h}$ are almost located on the smoothed surface $\widetilde{\manifold}_h$, and $F=F_\epsilon$ almost everywhere as the measure of $E_\delta$ is vanishing when $\epsilon\to 0$. 
	
	$(ii)$ For the second statement, we notice the following relation:
	\begin{equation}\label{eq:prop1}
	\begin{aligned}
	\norm{\partial \f r_{\tau_{h,j}} -  \partial \tilde{\f r}_{\tau_{h,j}}}_{\infty,\tau_{h,j}} \leq &\norm{\partial \f r_{\tau_{h,j}} - \bar{G}_h \Gamma^*_{j}  }_{\infty,\tau_{h,j}} + \norm{ \bar{G}_h  \Gamma^*_{j} -  \bar{G}_h \Gamma_{j}}_{\infty,\tau_{h,j}}\\
	& + \norm{ \bar{G}_h \Gamma_{j} -  \partial \tilde{\f r}_{\tau_{h,j}}}_{\infty,\tau_{h,j}}. 
	\end{aligned}
	\end{equation}
	Here we take $\tau_{h,j}$ the parameter domain for both $\f r_{\tau_{h,j}} $ and $\tilde{\f r}_{\tau_{h,j}}$.
	$   \Gamma^*_{j}$ and $   \Gamma_{j}$  are the local geometric transformations which are obtained from the linear interpolations of the geometric mappings $ \f r_{\tau_{h,j}} $ and $\tilde{\f r}_{\tau_{h,j}}$ at the vertices of $\tau_{h,j}$, respectively. Note that here we always chose $\tau_{h,j}$ to be the parameter domain, therefore $\Gamma_j \equiv \operatorname{Id}$ for all $j\in J_h$. $\bar{G}_h$ is the local PPR gradient recovery operator.
	
	The first and the third terms on the right-hand side of \eqref{eq:prop1} can be estimated using polynomial preserving properties of $\bar{G}_h$ and the smoothness of the functions $\f r_{\tau_{h,j}}$ and $\tilde{\f r}_{\tau_{h,j}}$ defined on every triangle $\tau_{h,j}$, which gives
	\begin{equation}\label{eq:prop2}
	\norm{\partial \f r_{\tau_{h,j}} -  \bar{G}_h \Gamma^*_{j} }_{\infty,\tau_{h,j}} \leq c_1 \norm{\f r_{\tau_{h,j}} }_{W^{3,\infty}(\tau_{h,j})} h^2,\quad
	\norm{ \bar{G}_h \Gamma_{j} -  \partial \tilde{\f r}_{\tau_{h,j}}}_{\infty,\tau_{h,j}}
	\leq c_2 \norm{ \tilde{\f r}_{\tau_{h,j}}}_{W^{3,\infty}(\tau_{h,j})} h^2 .
	\end{equation}
	Here $\bar{G}_h$ is realized using planar gradient recovery schemes componentwisely to $\f r_{\tau_{h,j}}$ (e.g., polynomial preserving or patch recovery methods) based on the local coordinates on $\tau_{h,j}$. 
	The second term on the right-hand side of \eqref{eq:prop1} can be estimated from Proposition \ref{prop:aux2} and the boundedness result of the planar recovery operator $ \bar{G}_h$ \cite[Theorem 3.2]{NagaZhang2004}: 
	\begin{equation}\label{eq:prop3}
	\norm{  \bar{G}_h \Gamma^*_{j}   -  \bar{G}_h \Gamma_{j}   }_{\infty,\tau_{h,j}} \leq C \norm{  \partial \Gamma^*_{j}   -  \partial \Gamma_{j}   }_{\infty,\tau_{h,j}}  \leq  c_3 h^2. 
	\end{equation}
	Since $ \norm{\f r_{\tau_{h,j}} }_{3,\tau_{h,j}} $ and $\norm{ \tilde{\f r}_{\tau_{h,j}}}_{3,\tau_{h,j}}$ both are uniformly bounded, combining \eqref{eq:prop2} and \eqref{eq:prop3}, we return to \eqref{eq:prop1} to have the estimate
	\[  \norm{\partial \f r_{\tau_{h,j}} -  \partial \tilde{\f r}_{\tau_{h,j}} }_{\infty,\tau_{h,j}}  = \mathcal{O}(h^2). \]
	
	$(iii)$ For the equivalence \eqref{eq:norm_equ} we can use the results in \cite[page 811]{Demlow2009}, which is helpful to show the equivalence on each triangle pairs of $\manifold_h$ and $\widetilde{\manifold}_h$, that is
	\[c_{j,1} \norm{\tilde{v}_h}_{W^{k,p}(\widetilde{\tau}_{h,j})} \leq \norm{v_h}_{W^{k,p}(\tau_{h,j})} \leq c_{j,2} \norm{\tilde{v}_h}_{W^{k,p}(\widetilde{\tau}_{h,j})},\]
	for some constants $\set{c_{j,1}}_{j\in J_h}>0$ and $\set{c_{j,2}}_{j\in J_h}>0$.
	The equivalence for functions defined on triangle pairs of $\tau_j$ and $\tau_{h,j}$ is similarly shown.
	Then we arrive the following
	\[\tilde{c}_{j,1} \norm{\tilde{v}_h}_{W^{k,p}(\widetilde{\tau}_{h,j})} \leq \norm{v}_{W^{k,p}(\tau_{j})} \leq \tilde{c}_{j,2} \norm{\tilde{v}_h}_{W^{k,p}(\widetilde{\tau}_{h,j})},\;\]
	with constants $\set{\tilde{c}_{j,1}}_{j\in J_h}>0$ and $\set{\tilde{c}_{j,2}}_{j\in J_h}>0$.
	Since $\manifold_h \to \manifold$ as $h\to 0$, we have $\widetilde{\manifold}_h \to \manifold$ as well.
	This tells that $\tilde{c}_{j,1} ,\; \tilde{c}_{j,2} \to 1$ as  $h\to 0$, which indicates that the constants $\set{\tilde{c}_{j,1} }_{j\in J_h}$ and $\set{\tilde{c}_{j,2} }_{j\in J_h}$ are uniformly bounded. Then we derive the equivalence in \eqref{eq:norm_equ}.
\end{proof}

In order to prove the superconvergence in the case when the vertices of $\manifold_h$ are not located exactly on $\manifold$, but in a $h^2$-neighborhood around it, we use the following estimate.
\begin{lemma}
	\label{lem:surface_closeness}
	Let Assumption \ref{ass:irregular} hold, and let $\widetilde{\manifold}_h$ be constructed from Proposition \ref{prop:appr_surfaces}. 
	Let $v\in W^{3,\infty}(\manifold) $, and let $\tilde{v}_h: =\tilde{T}_h v $ be pullback of $v$ from $\manifold$ to $\widetilde{\manifold}_h$:  $\tilde{v}_h(\tilde{\f r}_{h,i}(\zeta)) =v(\f r_i(\zeta))$ for every $\zeta \in \Omega_i$ and all $i\in I_h$, where $\tilde{\f r}_{h,i}:\Omega_i \to \widetilde{\manifold}_{h,i}$ and $\f r_i:\Omega_i \to \manifold_i$.
	Then the following estimate holds:
	\begin{equation}
	\label{eq:inter_estimate1}
	\norm{\nabla_g v -  (\tilde{T}_h)^{-1} \nabla_{\tilde{g}_h}\tilde{v}_h }_{0,\manifold}  \lesssim  h^2  \norm{\nabla_g v  }_{0,\manifold} .
	\end{equation} 
\end{lemma}
\begin{proof}
	Recall \eqref{eq:local_gradient} for the definition of gradient in the local parametric domain, particularly, we take the local parametric domain to be $\tau_{h,j}$. 
	Then we have for every $j\in J_h$ the following estimates
	\begin{equation}\label{eq:ineq_lem_4.2}
	\begin{aligned}
	&\norm{\nabla_g v -  (\tilde{T}_h)^{-1} \nabla_{\tilde{g}_h}\tilde{v}_h }^2_{0,\tau_j}  =\int_{\tau_{h,j}} \abs{\nabla \bar{v}\left((\partial \f r_{\tau_{h,j}})^\dag -  (\partial \tilde{\f r}_{\tau_{h,j}} )^\dag\right) }^2\sqrt{\det(\partial \f r_{\tau_{h,j}}  (\partial \f r_{\tau_{h,j}} )^\top)}\\
	\leq & \norm{\operatorname{I} -  \partial \f r_{\tau_{h,j}}  (\partial \tilde{\f r}_{\tau_{h,j}} )^\dag}_{\infty,\tau_{h,j}}^2\int_{\tau_{h,j}} \abs{\nabla \bar{v}(\partial \f r_{\tau_{h,j}} )^\dag }^2\sqrt{\det(\partial \f r_{\tau_{h,j}}  (\partial \f r_{\tau_{h,j}} )^\top)}\\
	=& \norm{\operatorname{I} -  \partial \f r_{\tau_{h,j}}  (\partial \tilde{\f r}_{\tau_{h,j}} )^\dag}_{\infty,\tau_{h,j}}^2 \norm{\nabla_g v  }^2_{0,\tau_j} .
	\end{aligned}
	\end{equation} 
	Using the estimate \eqref{eq:geometric_error2} from Proposition \ref{prop:appr_surfaces} and the fact that $\tilde{\f r}_{\tau_{h,j}}$ is regular thus $\partial \tilde{\f r}_{\tau_{h,j}}$ and its inverse are uniformly bounded, we derive the following:
	\[\norm{\operatorname{I} -  \partial \f r_{\tau_{h,j}}  (\partial \tilde{\f r}_{\tau_{h,j}} )^\dag}_{\infty,\tau_{h,j}}\lesssim \norm{  \partial \f r_{\tau_{h,j}} -  \partial \tilde{\f r}_{\tau_{h,j}} }_{\infty,\tau_{h,j}}    =\mathcal{O}(h^2)   \quad \text{ for all } j\in J_h .\]
	We go back to \eqref{eq:ineq_lem_4.2} with the above estimate. Then \eqref{eq:inter_estimate1} is proven by adding all the terms of $j\in J_h$ and taking the square root.
\end{proof}

Now we are ready to show the superconvergence of the gradient recovery on $\manifold_h$, which is considered to be an answer to the open question in \cite{WeiChenHuang2010}.
\begin{theorem}
	\label{thm:superconvergence_error_surface}
	Let Assumption \ref{ass:irregular} hold, and 
	$u\in  W^{3,\infty}(\manifold)$ be the solution of \eqref{eq:var_laplace}, and $u_h$ be the solution of \eqref{eq:d_var_laplace}. Then using the isoparametric gradient recovery scheme in Algorithm \ref{alg:isop_gradient}, we have
	\begin{equation}
	\label{eq:superconvergence_error_surface}
	\begin{array}{ll}
	\norm{\nabla_g u - T_h^{-1}G_h u_h}_{0,\manifold}\leq  & Ch^2\left( \sqrt{\mathcal{A}(\manifold)}  D(g,g^{-1})\norm{u}_{3,\infty,\manifold}  + \norm{f}_{0,\manifold} \right)   \\
	& + \quad C h^{1+\min \set{1,\sigma}}\left(\norm{u}_{3,\manifold}+\norm{u}_{2,\infty,\manifold} \right),
	\end{array}
	\end{equation}
where $\mathcal{A}$ is the area function, and $D(g,g^{-1})$ is some constant which depends on the Riemann metric tensor.
\end{theorem} 
\begin{proof}
	This is readily shown using the triangle inequality
	\begin{equation}\label{eq:overall_error}
	\norm{\nabla_g u - T_h^{-1}G_h u_h}_{0,\manifold}  \leq  \norm{\nabla_g u - (\tilde{T}_h)^{-1} \nabla_{\tilde{g}_h}\tilde{u}_h }_{0,\manifold}
	+ \norm{(\tilde{T}_h)^{-1} \nabla_{\tilde{g}_h}\tilde{u}_h - T_h^{-1}G_h u_h}_{0,\manifold}.
	\end{equation}
	The first part on the right hand side of \eqref{eq:overall_error} is estimated using Lemma \ref{lem:surface_closeness}:
	\begin{equation}\label{eq:geometry_error}
	\norm{\nabla_g u - (\tilde{T}_h)^{-1} \nabla_{\tilde{g}_h}\tilde{u}_h }_{0,\manifold} \lesssim   h^2 \norm{\nabla_g u  }_{0,\manifold}.
	\end{equation}
	Assumption \ref{ass:irregular}, Proposition \ref{prop:geometry} and Proposition \ref{prop:appr_surfaces} ensure that the geometric conditions of \cite[Theorem 5.3]{DonGuo20} is satisfied, i.e., the $\mathcal{O}(h^{2\sigma})$ irregular condition, and the vertices of $\manifold_h$ is located on $\widetilde{\manifold}_{h}$ which is $C^3$ smooth.
	Then the second term on the right hand side of \eqref{eq:overall_error} 
	is estimated using \cite[Theorem 5.3]{DonGuo20}.
	That gives
	\begin{equation*}
	\begin{array}{ll}
	\norm{  \nabla_{\tilde{g}_h}\tilde{u}_h -\tilde{T}_hT_h^{-1}G_h u_h}_{0,\widetilde{\manifold}_h} \le & \tilde{C} h^2\left( \sqrt{\mathcal{A}(\widetilde{\manifold}_h)}  \tilde{D}(\tilde{g},\tilde{g}^{-1})\norm{\tilde{u}_h}_{3,\infty,\widetilde{\manifold}_h}  + \norm{\tilde{f}}_{0,\widetilde{\manifold}_h} \right)  \\
	& + \tilde{C} \; h^{1+\min \set{1,\sigma}}\left(\norm{\tilde{u}_h}_{3,\widetilde{\manifold}_h}+\norm{\tilde{u}_h}_{2,\infty,\widetilde{\manifold}_h} \right).
	\end{array}
	\end{equation*}
	The equivalence relation from \eqref{eq:norm_equ} in Proposition \ref{prop:appr_surfaces} gives the estimate on $\manifold$, which is
	\begin{equation}
	\label{eq:recovery_ex_error}
	\begin{array}{ll}
	\norm{(\tilde{T}_h)^{-1}  \nabla_{g} u_h -T_h^{-1}G_h u_h}_{0,\manifold} \le & C h^2\left( \sqrt{\mathcal{A}(\manifold)}  D(g,g^{-1})\norm{u}_{3,\infty,\manifold}  + \norm{f}_{0,\manifold} \right)  \\
	& +  C \; h^{1+\min \set{1,\sigma}}\left(\norm{u}_{3,\manifold}+\norm{u}_{2,\infty,\manifold} \right).
	\end{array}
	\end{equation}
	Using embedding theorem that the right-hand side of \eqref{eq:geometry_error} can actually be bounded by the first term on the right-hand side of \eqref{eq:recovery_ex_error}.
	The proof is concluded by putting \eqref{eq:geometry_error} and \eqref{eq:recovery_ex_error} together.
\end{proof}

\section{Numerical results}
\label{sec:numerics}
In this section, we present numerical examples to verify the theoretical results.
Two types of surfaces are tested as benchmark examples.
The first type is the unit sphere, where we add artificial $\mathcal{O}(h^2)$ perturbation to the discretized mesh in order to match exactly our assumptions. With this, we are able to show the geometric approximation of $\manifold_{h}$ to $\manifold_{h}^*$.
The second one is a more general surface, where the vertices of its discretization mesh do not necessarily locate on the exact surface.

We shall consider two different members in the family of Algorithm \ref{alg:isop_gradient}:
(i) Parametric polynomial preserving recovery denoted by $G^{pppr}_h$, generalized from PPR method, and 
(ii) Parametric superconvergent patch recovery method denoted by $G^{pspr}_{h}$, a generalization of the renowned $ZZ$-scheme.
For the sake of simplifying the notation,  we define:
\begin{align*}
&De^* = \norm{T_h\nabla_g u -\nabla_{g_h} u_h }_{0, \manifold_h^*}, 
&De= \norm{T_h\nabla_g u -\nabla_{g_h} u_h }_{0, \manifold_h}, \\
&De_I^*= \norm{\nabla_{g_h} u_I - \nabla_{g_h} u_h}_{0, \manifold_h^*},
&De_I = \norm{\nabla_{g_h} u_I - \nabla_{g_h} u_h}_{0, \manifold_h},\\
&De_r^* = \norm{T_h\nabla_g u - G^{pppr}_{h} u_h}_{0, \manifold_h^*},
&De_r= \norm{T_h\nabla_g u - G^{pppr}_{h} u_h}_{0,  \manifold_h},\\
&De_{r_2}^* = \norm{T_h\nabla_g u - G^{pspr}_{h} u_h}_{0, \manifold_h^*},
&De_{r_2}= \norm{T_h\nabla_g u - G^{pspr}_{h} u_h}_{0,  \manifold_h};
\end{align*}
where $u_h$ is the finite element solution, $u$ is the analytical solution and $u_I$ is the linear finite element interpolation of $u$.
We also remind that $\manifold_{h}^*$ denotes the exact interpolation of $\manifold$.

\subsection{Examples on a sphere with deviations}
We test with numerical solutions of the Laplace-Beltrami equation on the unit sphere.
This is a good toy example since we can artificially design the deviations to the discretization and compare it to the adhoc results.
The right hand side function $f$ is chosen to fit the exact solution $u=x_1x_2$.
For the unit sphere, it is relatively  simple to generate interpolated triangular meshes, denoted by $\manifold_h^*$.

\subsubsection{Verification of geometric supercloseness}
In this test, we firstly artificially add $\mathcal{O}(h^2)$ perturbation along both normal and tangential directions at each vertex of $\manifold_h^*$ and the resulting deviated mesh is denoted by $\manifold_h^3$. The magnitude of the perturbation is chosen as  $h^2$.  We firstly to verify the geometric supercloseness property.  To do this, we take each triangular element in $ \manifold_h^*$ as the reference element and compute the linear transformation from the reference element to the corresponding triangular element in $\manifold_h^3$.  
The numerical errors are displayed in Table \ref{tab:tan2}.
Clearly, all the maximal norm errors decay at rates of $\mathcal{O}(h)$ and there is no geometric supercloseness. The observed first-order convergence rates match well with the theoretical result in the Proposition \ref{prop:geometry} since we add  $\mathcal{O}(h^2)$  perturbations in the tangential direction and the Assumption \ref{ass:irregular} is not fulfilled.

\begin{table}[htb!]
	\centering
	\caption{Difference of geometric quantities between the $\manifold_h^*$ and  $\manifold_h^3$}
	\renewcommand{\arraystretch}{1.2}
	\resizebox{\textwidth}{!}{
		\begin{tabular}{|c|c|c|c|c|c|c|}
			\hline 
			Dof&$\|\partial \Gamma-  \operatorname{Id} \|_{\infty}$&Order&$\norm{\sqrt{\det{g_\Gamma}}- 1}_\infty$&Order&$\norm{g_\Gamma -\operatorname{I}}_\infty$&Order\\ \hline
			162&6.09e-01&--&1.07e+00&--&1.07e+00&--\\ \hline
			642&4.09e-01&0.58&6.57e-01&0.71&6.57e-01&0.71\\ \hline
			2562&2.07e-01&0.98&3.32e-01&0.99&3.32e-01&0.99\\ \hline
			10242&1.00e-01&1.05&1.55e-01&1.09&1.55e-01&1.09\\ \hline
			40962&4.96e-02&1.01&7.54e-02&1.04&7.54e-02&1.04\\ \hline
			163842&2.48e-02&1.00&3.73e-02&1.02&3.73e-02&1.02\\ \hline
			655362&1.24e-02&1.00&1.85e-02&1.01&1.85e-02&1.01\\ \hline
			2621442&6.21e-03&1.00&9.21e-03&1.01&9.21e-03&1.01\\ \hline
	\end{tabular}}
	\renewcommand{\arraystretch}{1}
	\label{tab:tan2}
\end{table}

Then, we construct a deviated discrete surface that satisfies the Assumption \ref{ass:irregular}. To do this, we add $\mathcal{O}(h^2)$ perturbation along with normal directions and  $\mathcal{O}(h^3)$ perturbation along with tangential directions at each vertex of $\manifold_h^*$ and the resulting deviated mesh is denoted by $\manifold_h$. Similarly, we can compute the Jacobian $\partial\Gamma$, metric tensor $g_{\Gamma}$, and the determinant $\sqrt{\det{g_\Gamma}}$. The numerical results are documented in Table \ref{tab:tan3}.  As predicted by the Proposition \ref{prop:geometry}, we can observe the superconvergence of the geometric approximations in the above quantities. 

\begin{table}[htb!]
	\centering
	\caption{Difference of geometric quantities between the $\manifold_h^*$ and  $\manifold_h$}
	\renewcommand{\arraystretch}{1.2}
	\resizebox{\textwidth}{!}{
		\begin{tabular}{|c|c|c|c|c|c|c|}
			\hline 
			Dof&$\|\partial \Gamma-  \operatorname{Id} \|_{\infty}$&Order&$\norm{\sqrt{\det{g_\Gamma}}- 1}_\infty$&Order&$\norm{g_\Gamma -\operatorname{I}}_\infty$&Order\\ \hline
			162&2.02e-01&--&3.44e-01&--&3.44e-01&--\\ \hline
			642&8.61e-02&1.24&1.34e-01&1.36&1.34e-01&1.36\\ \hline
			2562&2.18e-02&1.99&3.60e-02&1.90&3.60e-02&1.90\\ \hline
			10242&5.48e-03&1.99&8.98e-03&2.01&8.98e-03&2.01\\ \hline
			40962&1.37e-03&2.00&2.25e-03&2.00&2.25e-03&2.00\\ \hline
			163842&3.43e-04&2.00&5.63e-04&2.00&5.63e-04&2.00\\ \hline
			655362&8.59e-05&2.00&1.41e-04&2.00&1.41e-04&2.00\\ \hline
			2621442&2.15e-05&2.00&3.52e-05&2.00&3.52e-05&2.00\\ \hline
	\end{tabular}}
	\renewcommand{\arraystretch}{1}
	\label{tab:tan3}
\end{table}

\begin{table}[htb!]
	\centering
	\caption{Difference of geometric quantities between the $\manifold_h^*$ and  $\manifold_h^4$}
	\renewcommand{\arraystretch}{1.2}
	\resizebox{\textwidth}{!}{
		\begin{tabular}{|c|c|c|c|c|c|c|}
			\hline 
			Dof&$\|\partial \Gamma-  \operatorname{Id} \|_{\infty}$&Order&$\norm{\sqrt{\det{g_\Gamma}}- 1}_\infty$&Order&$\norm{g_\Gamma -\operatorname{I}}_\infty$&Order\\ \hline
			162&2.66e-01&--&2.95e-01&--&4.12e-01&--\\ \hline
			642&2.12e-01&0.33&1.30e-01&1.20&1.66e-01&1.32\\ \hline
			2562&9.98e-02&1.09&3.49e-02&1.89&5.07e-02&1.71\\ \hline
			10242&5.79e-02&0.79&9.11e-03&1.94&1.24e-02&2.04\\ \hline
			40962&2.86e-02&1.02&2.24e-03&2.02&3.31e-03&1.90\\ \hline
			163842&1.45e-02&0.98&5.71e-04&1.97&8.25e-04&2.00\\ \hline
			655362&7.31e-03&0.99&1.46e-04&1.96&2.06e-04&2.01\\ \hline
			2621442&3.68e-03&0.99&3.71e-05&1.98&5.25e-05&1.97\\ \hline
	\end{tabular}}
	\renewcommand{\arraystretch}{1}
	\label{tab:normrand}
\end{table}
We also want to emphasize that the condition in \eqref{eq:close_condition2} is essential for Jacobian supercloseness. To demonstrate this, we consider a deviated mesh $\manifold_h^4$ which is constructed by adding $rand\times h^2$ perturbation along with normal directions and  $h^3$ perturbation along with tangential directions at each vertex of $\manifold_h^*$.  We repeat the same computation and report the numerical results in Table \ref{tab:normrand}. Clearly, we can observe the $\mathcal{O}(h^2)$ geometric supercloseness in both metric tensor and determinant of the metric tensor. However, only $\mathcal{O}(h)$ approximation results can be observed for the Jacobian $\partial\Gamma$ which again confirms the Proposition \ref{prop:geometry}.

\subsubsection{Superconvergence of gradient recovery on deviated sphere}

Now, we show the superconvergence of gradient recovery on the deviated sphere with the property of geometric supercloseness.
We solve the Laplace-Beltrami equation on both $\manifold_h^*$ and $\manifold_h$ and the numerical performances are tabulated in Table \ref{tab:sphere}.  For the finite element gradient error,  the expected optimal convergence rate $\mathcal{O}(h)$ can be observed on both the meshes $\manifold_{h}^*$ and $\manifold_{h}$. 
We concentrate on the finite element supercloseness error.  
We can observe  $\mathcal{O}(h^{2})$ supercloseness on $\manifold_h^*$ and   $\mathcal{O}(h^{1.94})$ supercloseness $\manifold_h$.  It gives solid evidence that the $\mathcal{O}(h^{2\sigma})$ irregular condition also holds true for the perturbed mesh $\manifold_h$.
For the recovered gradient error,  we observed almost the same $\mathcal{O}(h^2)$ superconvergence rates on both discretized surfaces using isoparametric gradient recovery schemes,  which validates Theorem \ref{thm:superconvergence_error_surface}.

\begin{table}[htb!]
	\centering
	\caption{Numerical results of solving Laplace-Beltrami equation on the sphere}
	\renewcommand{\arraystretch}{1.2}
	\resizebox{\textwidth}{!}{
		\begin{tabular}{|c|c|c|c|c|c|c|c|c|}
			\hline 
			Dof&$De^*$&Order&$De^*_{I}$&Order&$De^*_{r}$&Order&$De^*_{r_2}$&Order\\ \hline\hline
			162&5.13e-01&--&9.71e-01&--&5.25e-01&--&5.25e-01&--\\ \hline
			642&1.96e-01&1.40&9.58e-03&6.71&5.41e-02&3.30&5.40e-02&3.30\\ \hline
			2562&9.83e-02&1.00&2.60e-03&1.88&1.37e-02&1.98&1.37e-02&1.98\\ \hline
			10242&4.92e-02&1.00&6.97e-04&1.90&3.45e-03&1.99&3.48e-03&1.98\\ \hline
			40962&2.46e-02&1.00&1.85e-04&1.92&8.67e-04&1.99&8.86e-04&1.97\\ \hline
			163842&1.23e-02&1.00&4.86e-05&1.93&2.17e-04&2.00&2.29e-04&1.95\\ \hline
			655362&6.15e-03&1.00&1.27e-05&1.93&5.45e-05&2.00&6.05e-05&1.92\\ \hline
			2621442&3.07e-03&1.00&3.32e-06&1.94&1.37e-05&2.00&1.67e-05&1.86\\ \hline		
			Dof&$De$&Order&$De_{I}$&Order&$De_{r}$&Order&$De_{r_2}$&Order\\ \hline\hline
			162&4.87e-01&--&7.89e-01&--&4.75e-01&--&4.74e-01&--\\ \hline
			642&2.12e-01&1.21&1.06e-01&2.91&2.11e-02&4.53&2.14e-02&4.50\\ \hline
			2562&1.00e-01&1.08&2.66e-02&2.00&5.21e-03&2.02&5.32e-03&2.01\\ \hline
			10242&4.94e-02&1.02&6.67e-03&1.99&1.32e-03&1.99&1.39e-03&1.93\\ \hline
			40962&2.46e-02&1.01&1.68e-03&1.99&3.34e-04&1.98&3.83e-04&1.86\\ \hline
			163842&1.23e-02&1.00&4.20e-04&2.00&8.50e-05&1.98&1.11e-04&1.79\\ \hline
			655362&6.15e-03&1.00&1.05e-04&2.00&2.16e-05&1.98&3.41e-05&1.70\\ \hline
			2621442&3.07e-03&1.00&2.63e-05&2.00&5.48e-06&1.98&1.10e-05&1.63\\ \hline
	\end{tabular}}
	\renewcommand{\arraystretch}{1}
	\label{tab:sphere}
\end{table}

The numerical results of using isoparametric SPR show superconvergence rate $\mathcal{O}(h^{1.9})$ and $\mathcal{O}(h^{1.7})$ on $\manifold_h^*$ and $\manifold_h$, respectively, which indicates that it is also a valid algorithm.

\subsection{The Laplace-Beltrami equation on general surface} 
In this example,  we consider a general surface which can be represented as the zero level-set of the following function 
\begin{equation*}
\Phi(x) = (x^2 - 1)^2 + (y^2 - 1)^2 + (z^2 - 1)^2 - 1.05.
\end{equation*}
We solve  the Laplace-Beltrami type equation
\begin{equation*}
-\Delta_g u + u = f;
\end{equation*}
with  the exact solution $u =\exp(x^2+y^2+z^2)$.
The right hand side function $f$ can be  computed from $u$.

The initial mesh of the general surface was generated using the three-dimensional surface mesh generation module of the Computational Geometry Algorithms Library \cite{cgal}.
To get meshes in other levels, we first perform the uniform refinement. Then we project the newest vertices onto $\manifold$.
In the general case, there is no explicit projection map available.
Hence we adopt the first-order approximation of projection map as given in \cite{DemlowDziuk2007}.
Therefore, the vertices of the meshes are not on the exact surface $\manifold$ but in a $\mathcal{O}(h^2)$ neighborhood along the normal vectors.

\begin{figure}[ht!]
	\centering
	\includegraphics[width=0.4\textwidth]{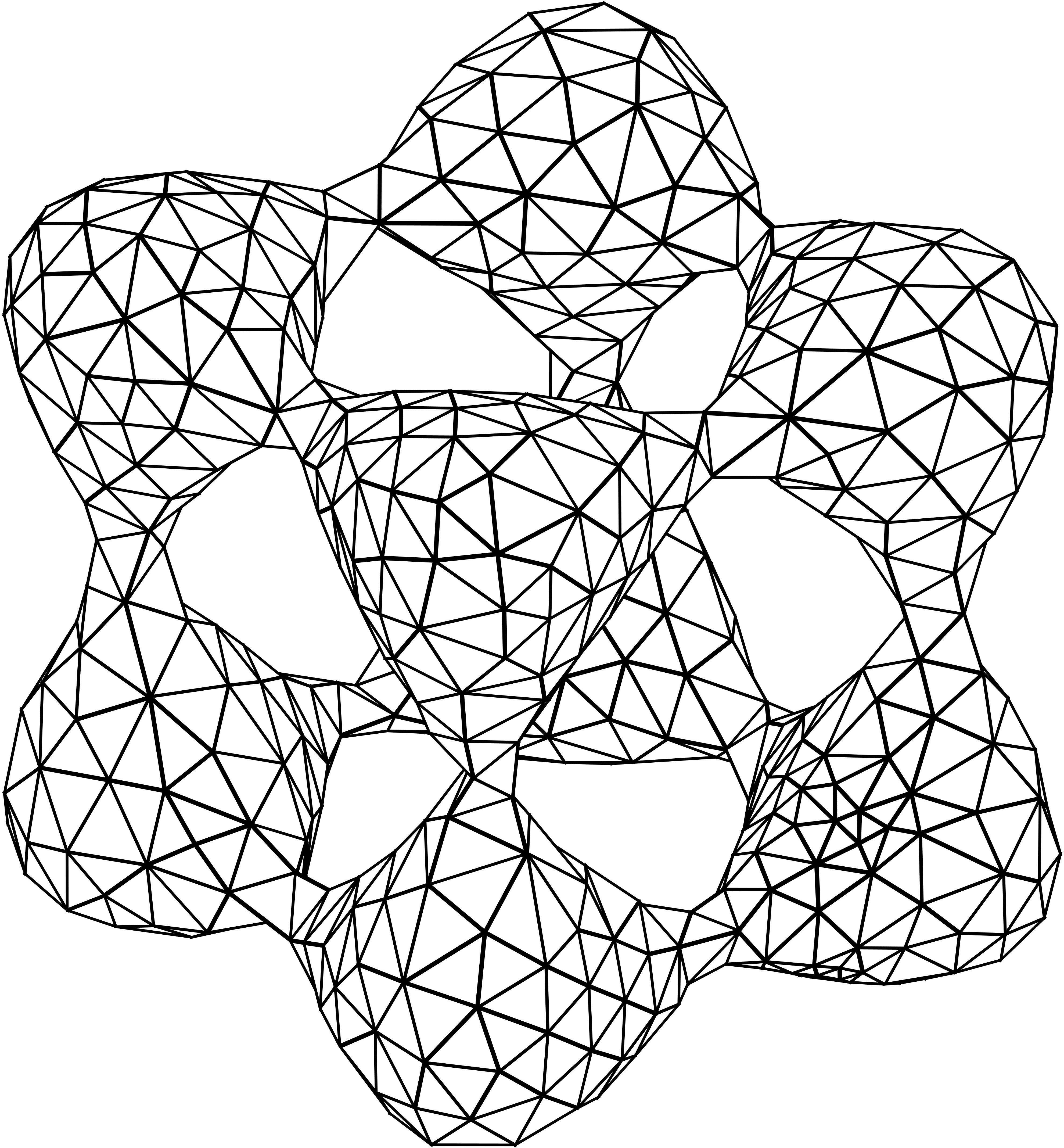}
	\caption{Initial mesh on a general surface}
	\label{fig:quartic}
\end{figure}

\begin{table}[ht]
	\centering
	\caption{Numerical results of solving Laplace-Beltrami equation on a general surface}
	\renewcommand{\arraystretch}{1.2}
	\resizebox{\textwidth}{!}{
		\begin{tabular}{|c|c|c|c|c|c|c|c|c|}
			\hline 
			Dof&$De$&Order&$De_{I}$&Order&$De_{r}$&Order&$De_{r_2}$&Order\\ \hline\hline
			701&1.32e+01&--&5.67e+00&--&1.12e+01&--&1.19e+01&--\\ \hline
			2828&6.93e+00&0.93&1.67e+00&1.75&3.66e+00&1.61&3.68e+00&1.69\\ \hline
			11336&3.52e+00&0.98&4.89e-01&1.77&1.04e+00&1.81&1.06e+00&1.80\\ \hline
			45368&1.77e+00&0.99&1.34e-01&1.87&2.76e-01&1.91&2.88e-01&1.87\\ \hline
			181496&8.86e-01&1.00&3.54e-02&1.92&7.12e-02&1.96&7.77e-02&1.89\\ \hline
			726008&4.43e-01&1.00&9.21e-03&1.94&1.81e-02&1.98&2.15e-02&1.86\\ \hline
	\end{tabular}}
	\renewcommand{\arraystretch}{1}
	\label{tab:quartic}
\end{table}
The first level of mesh is plotted in Figure \ref{fig:quartic}.
The history of numerical errors is documented in Table \ref{tab:quartic}.
As expected, we can observe the $\mathcal{O}(h)$ optimal convergence rate for the finite element gradient.
The rate of $\mathcal{O}(h^{1.9})$ can be observed for the error between the finite element gradient and the gradient of the interpolation of the exact solution.
Again, it means the mesh $\manifold_h$ satisfies the $\mathcal{O}(h^{2\sigma})$ irregular condition.
As depicted by Theorem \ref{thm:superconvergence_error_surface}, the recovered gradient using parametric polynomial preserving recovery is superconvergent to the exact gradient at the rate of $\mathcal{O}(h^2)$ even though the vertices are not located on the exact surface.  For parametric superconvergent patch recovery,  it deteriorates a litter bit but we can still observe $\mathcal{O}(h^{1.86})$ superconvergence.

\vspace{2em}
\section{Conclusion} 
\label{sec:conclusion}

In this paper,  we have established a new concept called geometric supercloseness which is fundamental for analyzing the superconvergence of differential structure on surfaces.
It is important for the numerical analysis on deviated surfaces, particularly when there has no information of the exact geometry.
In such cases, the vertices are typically not necessarily located on the underlying exact surfaces, and the exact normal vectors (or tangent spaces) are not known. 
We focused on investigating the post-processing of numerical solutions, where an algorithmic framework for superconvergent gradient recovery methods on deviated discretized surfaces is proposed and analyzed.
It concluded some open questions existing in the literature.

Although we have only studied the condition for piece-wise linear approximations,  the higher-order cases are of course interesting and meaningful to investigate as well, particularly the discretization assumptions that lead to higher-order geometric supercloseness.
On the other hand, we have concentrated on scalar functions in this paper. For those problems whose solutions are tangent vector fields on surfaces, we would wish to have similar error analysis with deviated surfaces as well. 
In that case, we need to investigate the superconvergence of curvature terms involving higher-order differential structures, then the higher-order geometric supercloseness would be interesting for future study.

\vspace{2em}

\noindent {\bf Declarations.} 
The authors have no relevant financial or nonfinancial interests to disclose.

\vspace{2em}

\noindent {\bf Data Availability.} 
The datasets generated and/or analyzed during the current study are available from the authors on reasonable request.

\vspace{2em}

\noindent {\bf Acknowledgments.} 
The authors thank the referees for their patience and comments which have helped to improve the presentation of the paper.
The authors acknowledge Dr. Ao Sun for some helpful discussion on the Whitney Approximation Theorem.
The work of GD was partially  supported by an NSFC grant No. 12001194, and an NSF grant of Hunan Province No. 2024JJ5413. The work of HG was partially supported by the Andrew Sisson Fund, Dyason Fellowship, and the Faculty Science Researcher Development Grant of the University of Melbourne.
The work of TG was partially supported by an NSFC grant No. 12101228, and an Innovative Platform Project of Hunan Province No. 20K078.

\bibliographystyle{plain}

\end{document}